\documentclass[a4paper,fleqn]{cas-sc}

\usepackage[authoryear]{natbib}

\def\tsc#1{\csdef{#1}{\textsc{\lowercase{#1}}\xspace}}
\tsc{WGM}
\tsc{QE}

\usepackage{amsmath,amsthm,amsfonts,amssymb}  
\usepackage{mathtools}
\usepackage{hyperref}
\usepackage{cleveref}
\usepackage{xcolor}
\usepackage{enumitem}

\newtheorem{theorem}{Theorem}[section]

\crefname{Proposition}{proposition}{propositions}
\Crefname{Proposition}{Proposition}{Propositions}

\crefname{Corollary}{corollary}{Corollaries}
\Crefname{Corollary}{Corollary}{Corollaries}

\crefname{Result}{result}{results}
\Crefname{Result}{Result}{Results}

\newtheorem{lemma}{Lemma}[section]
\crefname{Lemma}{lemma}{lemmas}
\Crefname{Lemma}{Lemma}{Lemmas}

\crefname{Remark}{remark}{remarks}
\Crefname{Remark}{Remark}{Remarks}

\newtheorem{definition}{Definition}[section]
\crefname{Definition}{definition}{definitions}
\Crefname{Definition}{Definition}{Definitions}

\crefname{Example}{example}{examples}
\Crefname{Example}{Example}{Examples}

\newcommand{\Cov}{\mathsf{Cov}}
\newcommand{\trace}{\mathsf{trace}}
\newcommand{\Var}{\mathsf{Var}}
\newcommand{\E}{\mathsf{E}}
\newcommand{\D}{\mathsf{D}}
\newcommand{\A}{\mathsf{A}}
\newcommand{\B}{\mathsf{B}}
\newcommand{\Eof}[1]{\E\left(#1\right)}
\newcommand{\prob}{\mathsf{P}}

\newcommand{\cT}{{\mathcal T}}

\newcommand{\cM}{{\mathcal M}}

\newcommand{\cB}{{\mathcal B}}

\newcommand{\cA}{{\mathcal A}}
\newcommand{\cE}{{\mathcal E}}
\newcommand{\cF}{{\mathcal F}}

\newcommand{\ve}{\varepsilon}
\newcommand{\mbR}{{\mathbb R}}

\newcommand{\mbH}{{\mathbb H}}

\newcommand{\llim}[1]{\mathop{L^2\text{-}\lim}\limits_{#1}\,}
\newcommand{\lplim}[1]{\mathop{L^p\text{-}\lim}\limits_{#1}\,}

\newcommand{\myExp}[1]{\exp\left(#1\right)}
\newcommand{\differential}[1]{\mathrm{d}#1}

\newcommand{\norm}[1]{\| #1 \|}
\newcommand{\dist}{\mathsf{d}}

\begin{document}
\let\WriteBookmarks\relax
\def\floatpagepagefraction{1}
\def\textpagefraction{.001}

\shorttitle{Self-intersection local times in stochastic flows}    

\shortauthors{Izyumtseva and KhudaBukhsh}  

\title [mode = title]{Self-intersection local times for Volterra Gaussian processes in stochastic flows with interaction}  



%

\author[]{Olga Izyumtseva}[orcid=0009-0002-4395-2020]



\ead{Olga.Iziumtseva1@nottingham.ac.uk}

\ead[url]{https://www.nottingham.ac.uk/mathematics/people/olga.iziumtseva1}

\credit{Conceptualization of this study, Methodology, Writing}


\author[]{Wasiur R. KhudaBukhsh}[orcid=0000-0003-1803-0470]

\cormark[1]


\ead{wasiur.khudabukhsh@nottingham.ac.uk}

\ead[url]{https://www.wasiur.xyz/}

\credit{Methodology, Writing}

\affiliation[]{organization={School of Mathematical Sciences, University of Nottingham},
            addressline={University Park}, 
            city={Nottingham},
            postcode={NG7 2RD}, 
            state={Nottinghamshire},
            country={United Kingdom}}

\cortext[1]{Corresponding author}










\begin{abstract}
  In this paper, we study self-intersection local times for a stochastic process $x(u(\cdot),t)$, where $u$ is a Gaussian process of the form $u(t)=\int^t_0k(t,s)\differential{w(s)}$, $k$ is a deterministic kernel of the  Volterra type, $w$ is a Wiener process, and $x$ is a solution to the \emph{equation with interaction}. Equations with interaction are a class of interacting particle system described by stochastic differential equations whose coefficients depend on a random measure (initial distribution of particles) transformed by the flow of solutions. Considering the occupation measure of $u$ as the initial condition for the equation with interaction allows us to define a stochastic flow with interaction driven by self-intersection local times of the process $u$. The study of such stochastic differential equations whose coefficients carry information about the geometric properties of curves is new. They previously appeared only for deterministic differential equations and smooth curves, where the geometric characteristics typically considered are length, curvature, and so on. In this paper, we prove the existence of multiple self-intersection local times for the process $x(u(\cdot),t)$ and establish a ``change of variable formula" that allows us to describe self-intersection local times for the process $x(u(\cdot),t)$ in terms of the weighted self-intersection local times for the process $u.$ We describe the corresponding asymptotics of the self-intersection local times for $x(u(\cdot),t)$ for large $t$. Moreover, the existence of weighted self-intersection local times is established for a large class of unbounded weights, which is of independent interest. 
\end{abstract}




\maketitle

\tableofcontents

\section{Introduction}
\label{sec:introduction}
Our primary objective in this paper is to study geometric characteristics of Volterra Gaussian processes in  stochastic flows with interaction. The interest in this problem is motivated by two fascinating questions that have been actively studied for more than forty years in analysis and probability theory literature. The first question is related to the evolution of geometric characteristics for smooth curves and Borel sets under the action of stochastic flows (for example, see \cite{carverhill1985flows}, \cite{le1985isotropic}, \cite{baxendale1986isotropic}, \cite{zirbel1997mass}, \cite{zirbel1997translation}, \cite{dimitroff2009dispersion}, \cite{vadlamani2006global}).
The literature on the evolution of geometric characteristics for non-smooth random curves such as the Wiener process is rather sparse. This is because defining appropriate geometric characteristics for non-smooth random curves is itself a nontrivial technical task. The early works of \cite{varadhan1969appendix}, \cite{rosen1986renormalized}, \cite{dynkin1988regularized}, and in particular, \cite{le1990wiener} lay the foundation of self-intersection local times as the fitting mathematical object to describe the geometry of non-smooth random curves. 
The second question arises from the desire to construct an evolution equation whose coefficients are driven by the geometric characteristics of an evolving object. A number of authors have studied the evolution of randomly moving curves and manifolds, see, for example, \cite{funaki1983random}, \cite{kardar1986dynamic}. However, the proposed evolutions do not take into account the changing of geometric characteristics of an evolving curve. One of the first attempts to take this into account is \cite{dorogovtsev2019self}. In the deterministic context, the time evolution of a knot is often described by the filament equation whose coefficients are driven by the curvature.  We refer the interested readers to \cite{arnold2009topological} for a detailed discussion on this topic. In this paper, we will use the self-intersection local times as a proxy for the geometry of a random curve, and study an (infinite-dimensional) stochastic differential equation whose coefficients are driven by the said geometry. 






Given a fixed positive integer $d$, let $W$ be a Brownian sheet defined on the Borel subsets of $[0,\infty)\times\mathbb{R}^d$ with  finite Lebesgue measure (see \cite[Example 3.5]{Lifshits2012LecturesOnGaussianProcesses}) and let $\mu_0$ be a (possibly random) probability measure on $\cB(\mathbb{R}^d),$ the Borel $\sigma$-algebra on $\mathbb{R}^d.$ 
The following stochastic differential equation (SDE) introduced in  \cite{dorogovtsev2023measure} 
\begin{align}
\label{eq:interaction}
\begin{cases}
\differential{x}(u,t)=a(x(u,t),\mu_t)\differential{t}+\int_{\mbR^d}b(x(u,t),\mu_t,z)W(\differential{t},\differential{z}),\\	
x(u,0)=u,\ u\in\mbR^d,\\
\mu_t=\mu_0\circ x(\cdot,t)^{-1},
\end{cases}
\end{align}
is called the  \emph{equation with interaction}. It was proved in \cite[Theorem 2.3.1, p. 43]{dorogovtsev2023measure} that if the coefficients satisfy a Lipschitz condition with respect to the spatial and the measure-valued variables, then there exists the unique solution to \eqref{eq:interaction}. Moreover, if the coefficients are two times continuously differentiable with respect to the spatial variable, then, for each $t\geqslant 0, $  the $x(\cdot,t):\mbR^d\to\mbR^d$ is a diffeomorphism, almost surely \citep[Theorem 6.3.1]{dorogovtsev2023measure}.
For other properties of the equation with interaction, such as intermittency, ergodicity, Krylov--Veretennikov expansion, we refer the interested reader to \cite{dorogovtsev2023intermittency}, \cite{chen2024exponential}, \cite{dorogovtsev2025ergodic}, \cite{dorogovtsev2026krylov}
and the references therein. Our aim is to study the geometric characteristics of an $\mbR^d$-valued Volterra Gaussian process in the stochastic flow generated by the equation with interaction. 

Let $\{u(t): t\in[0,1]\}$ be an $\mbR^d$-valued centred Gaussian process. It follows from Le Gall's work on the asymptotic expansion of the Wiener sausage \citep{le1990wiener} that the following formal expression, called the $k$-multiple self-intersection local time of the process $u,$
 \begin{align}
\label{eq:SILT_formal}
T_{k}^{u}=\int_{\Delta_k}\left(\prod^{k-1}_{i=1}\delta_0(u(t_{i+1})-u(t_i))\right)\differential{t_1}\cdots \differential{t_k},
 \end{align}
 where  $\Delta_k=\{t_1,\ldots, t_k\in[0,1]:\ 0\leqslant t_1\leqslant\ldots\leqslant t_k\leqslant 1\},\ k\geqslant 2,$ and $\delta_0$ is the $d$-dimensional Dirac delta function at zero, can be considered as geometric characteristics of the process $u.$ The quantity $T_{k}^{u}$ registers times $t_1\neq t_2\neq\cdots\neq t_k$
 such that $u(t_1)=u(t_2)=\cdots=u(t_k),$ i.e., $T_{k}^{u}$ measures the amount of time the process $u$ spends in small neighbourhoods of its self-intersection points of multiplicity $k$. 
 Self-intersection local times being the geometric characteristics of continuous non-smooth random processes have widespread  applications; see \cite{varadhan1969appendix}, \cite{jung2014tanaka}, \cite{westwater1982edwards}, \cite{van2001self}, \cite{van2003large}, \cite{chen2019parabolic}, \cite{denHollander2009random}, and discussions and references therein.
 The existence of multiple points for random processes is a nontrivial question. For a Wiener process, it was studied in \cite{dvoretzky1954multiple,dvoretzky1950double,dvoretzky1958points,dvoretzky1957triple}, for Gaussian random fields including fractional Brownian fields, and solutions to the systems of stochastic heat and wave equations, it was studied in \cite{xiao2002local}, \cite{goldman1981points}, \cite{kono1978double}, \cite{rosen1984self}, \cite{talagrand1998multiple}, \cite{dalang2012critical}, \cite{dalang2015multiple}, \cite{dalang2021multiple}.

To give a rigorous meaning to the formal expression in  \eqref{eq:SILT_formal}, consider a family of approximations $\{f_{\ve} :  \ve >0 \}$ that weakly converges to  $\delta_0$ as $\ve\to0$ in the sense that 
for any continuous bounded 
 function $\phi:\mathbb{R}^d\to\mathbb{R}$, we have 
 $$
 \int_{\mathbb{R}^d}\phi(z)f_{\ve}(z)\differential{z}\to \phi(0),
  $$
as $\ve \to 0$ 
and define approximations for the self-intersection local time in \eqref{eq:SILT_formal} of $u$ as follows
$$
T_{\ve,k}^{u}=\int_{\Delta_k}\left(\prod^{k-1}_{i=1}f_\ve(u(t_{i+1})-u(t_i))\right)\differential{t_1}\ldots \differential{t_k}.
$$
\begin{definition}
\label{def:SILT} The random variable 
$
T^{u}_k \coloneq \lplim{\ve\to0} T^{u}_{\ve,k}
$
is said to be the $k$-multiple self-intersection local time for $u$  if the limit exists for some $p\geqslant 2.$ If the limit exists, then the random variable $T^{u}_k$ does not depend on the choice of the approximating family $\{f_{\ve} :  \ve >0 \}$.
\end{definition}



\noindent
\vspace{0.25em}
\fcolorbox{black}{black!05}{
  \parbox{0.97\textwidth}{%
        Assume that $x$ is a stochastic flow of random diffeomorphisms satisfying \eqref{eq:interaction}, and $u$ is a Volterra Gaussian process. {We are interested in the geometry of the image $x(u, t) \equiv \{ x(u(s), t) : s \in [0, 1]\}$ of the process $u$ under the random diffeomorphism $x(\cdot,t)$ for $t>0$}. With this in mind, we adopt Le Gall's approach to geometry of random curves, and prove the existence of self-intersection local times for the process $\{x(u(s),t) :  s\in[0,1]\},$ for each $t>0$. We also describe its asymptotics for large $t$. 
    }%
}
\vspace{0.25em}

The asymptotics of self-intersection local times for stochastic processes is the fascinating question by itself, and is being actively pursued. See \cite{chen2004large}, \cite{chen2008intersection}, \cite{chen2023exponential}, \cite{lyu2025exponential}, \cite{dorogovtsev2019self}, \cite{gartner2000moment} to get a glimpse of some recent developments.
In order to construct self-intersection local times for the process $\{x(u(s),t): s\in[0,1]\}$, we need the notion of weighted self-intersection local times formally defined as
 \begin{align}
\label{eq:weighted_SILT_formal}
T_{k}^{u}(\rho) =\int_{\Delta_k}\rho(u(t_1))\prod^{k-1}_{i=1}\delta_0(u(t_{i+1})-u(t_i))\, \differential{t_1}\cdots \differential{t_k},
 \end{align}
 where $\rho:\mbR^d\to\mbR$ is a weight function. To motivate this notion, let us consider the process $F(u(s)),\ s\in[0,1]$, where $F:\mbR^d\to\mbR^d$ is a deterministic diffeomorphism mapping $y \in \mathbb{R}^d$ to $F(y) \equiv (F_1(y), \ldots, F_d(y)) \in \mathbb{R}^d$. Now, consider the following family of approximations for $\delta_0$
$$
\frac{1}{|\det \D F(F^{-1}(v_1))|^{k-1}}\prod^{k-1}_{i=1}f_{\ve}(F^{-1}(v_{i+1})-F^{-1}(v_{i})),\ v_1,\ldots,v_k\in\mbR^d,
$$
where 
$$
f_{\ve}(y)=\frac{1}{(2\pi\ve)^{\frac{d}{2}}}\myExp{-\frac{\|y\|^2}{2\ve}},\ \ve>0,\ y\in\mbR^d, 
$$
and
$\D F(y)\coloneq (\partial_jF_i(y))^d_{i, j=1},\ y\in\mbR^d$ denotes the Jacobian matrix of $F$ at $y$.  
Then, the analogous approximations for self-intersection local time of the process $F(u(s)),\ s\in[0,1]$ have the following representation
\begin{align*}
&{}T^{F(u)}_{\ve,k}=\int_{\Delta_k}\frac{1}{|\det \D F(u(t_1))|^{k-1}}\prod^{k-1}_{i=1}f_{\ve}(u(t_{i+1})-u(t_i))\, \differential{t}_1\cdots \differential{t}_k
=T^u_{\ve,k}\left(\frac{1}{|\det \D F|^{k-1}}\right).
\end{align*}
Therefore, the approximations for the self-intersection local times of the process $F(u(s)),\ s\in[0,1]$ are precisely the approximations for the weighted self-intersection local times of the process $u(s),\ s\in[0,1]$ with the weight function 
$$
\rho(z)=\frac{1}{|\det \D F(z)|^{k-1}},\ z\in\mbR^d,
$$
and, if the limit exists for some $p\geqslant 2$, then
$$ 
T^{F(u)}_{k} \coloneq \lplim{\ve\to0} T^{F(u)}_{\ve,k}=\lplim{\ve\to0} T^{u}_{\ve,k}\left(\frac{1}{|\det\D F|^{k-1}}\right).
$$

Thus, to work with the stochastic processes of the form $F(u(s)),\ s\in[0,1],$ where $F$ is diffeomorphism either random or deterministic, we naturally come to the notion of weighted self-intersection local times rigorously defined as follows. 
Again, consider a family of approximations $\{f_{\ve}:\ \ve>0\}$ that weakly converges to  $\delta_0$ as $\ve\to0$,
and define approximations for the weighted self-intersection local time in \eqref{eq:weighted_SILT_formal} of $u$ as follows
\begin{align}
  \label{eq:weighted_SILT_approx}
T_{\ve,k}^{u}(\rho)=\int_{\Delta_k}\rho(u(t_1))\prod^{k-1}_{i=1}f_\ve(u(t_{i+1})-u(t_i))\, \differential{t_1}\cdots \differential{t_k}.
\end{align}
\begin{definition}
\label{def:SILT} The random variable 
$
T^{u}_k(\rho)\coloneq \lplim{\ve\to0} T^{x}_{\ve,k}(\rho)
$
is said to be the weighted $k$-multiple self-intersection local time for $u$ with weight function $\rho$, if the limit exists for some $p\geqslant 2.$ 
\end{definition}

The construction of weighted self-intersection local times for a planar Wiener process was proposed by E. B. Dynkin in \cite{dynkin1988regularized}.  \cite{dorogovtsev2019hilbert} extended 
Dynkin's construction to Hilbert-valued weights.
Self-intersection local times for planar diffusion processes were studied in \cite{izyumtseva2008renormalization}, where it was proved that for a diffusion process $\{y(t) :  t\in[0,1]\}$ in ${\mathbb R}^2$ satisfying the stochastic
differential equation
$$
\begin{cases}
\differential{y}(t)=a(y(t))\differential{s}+\sigma(y(t))\differential{w}(t),\\
y(0)=y_0,
\end{cases}
$$
with a nondegenerate diffusion $\sigma$, the self-intersection local times for $y$ can be constructed using the weight function
$$
\rho(z)=\frac{1}{|\det \sigma(z)|^{k-1}},\ z\in\mathbb {R}^2.
$$
In order to study the geometric characteristics of the process $\{x(u(s),t): s\in[0,1]\}$, we need the occupation measure, which we define now.
\begin{definition}
\label{def:Occupation measure} The random measure $\mu$ on $\cB(\mbR^d)$ defined as
$
\mu(A)=\int^1_01_{A}(u(t))\differential{t},\ A\in\cB(\mbR^d),
$
is said to be the occupation measure of the process $u.$
\end{definition}
For each bounded and measurable function $\phi:\mbR^d\to\mbR$, the following occupation measure formula holds
$$
\int^1_0\phi(u(t))\differential{t}=\int_{\mbR^d}\phi(y)\mu(\differential{y}).
$$
Moreover, for $k \ge 2$ and for a bounded and measurable function $\phi:\mbR^{dk}\to\mbR$, 
\begin{align*}
\int^1_0\cdots\int^1_0\phi(u(t_1),\ldots,u(t_{k}))\, \differential{t_1}\cdots\differential{t_k}
&{}=\int_{\mbR^d}\cdots\int_{\mbR^d}\phi(y_1,\ldots,y_k)\mu(\differential{y_1})\cdots\mu(\differential{y_k}).
\end{align*}
Therefore, if the random variable $T^u_k$ exists, then applying the occupation formula for the approximations of delta function and passing to the limit, one can conclude that
\begin{align}
\label{eq:occupation measure representation}
\int^1_0\cdots\int^1_0\prod^{k-1}_{i=1}\delta_0(u(t_{i+1})-u(t_i))\,\differential{t}_1\cdots\differential{t}_k
&
{}=\int_{\mbR^d}\cdots\int_{\mbR^d}\prod^{k-1}_{i=1}\delta_0(y_{i+1}-y_i)\mu(\differential{y_1})\cdots \mu(\differential{y_k}).
\end{align}
Since one  obtains the self-intersection local times by integrating products of delta functions with respect to the occupation measure, the occupation measure itself can be seen as describing the geometry of the random curve $u$ in the sense of Le Gall.


Let us describe our approach to defining a stochastic flow with the interaction driven by the geometric characteristics of a random curve. Assume that $u$ and $W$ are independent. 
Consider the equation with interaction \eqref{eq:interaction}, where the initial measure $\mu_0$ is the occupation measure of the process $u$. Then, the measure-valued stochastic process $\mu_t$ is precisely the occupation measure of the process $\{x(u(s),t) : s\in[0,1]\},$ for each fixed $t$, and hence, following the previous discussion, describes its geometry. Note that, due to \eqref{eq:occupation measure representation}, we in fact obtain a stochastic differential equation driven by the geometric characteristics of process $\{x(u(s),t):  s\in[0,1]\}$. Such stochastic differential equations driven by geometric characteristics of nonsmooth random curves are a relatively new  model, understudied before.

The equation with interaction is indeed an interacting particle system, where $x(u(s),t)$ is the position  at time $t$ of the particle starting at $u(s)$. Moreover, the motion of the particle at  $u(s)$ of our random curve depends on the motion of all other particles starting at $u(s),\ s\in[0,1]$ through the occupation measure in the coefficients of the stochastic differential equation with interaction. However, it is crucial to note that it is not the standard McKean--Vlasov equation since $\mu_t$ is a random measure describing the distribution of mass of particles at time $t$, and not the probability law of random diffeomorphism $x(\cdot, t): \mbR^d \to \mbR^d$. We refer the interested readers to \cite{Sznitman1991Chaos,Chaintron2022reviewI,Chaintron2022reviewII} for a review of the important topics of the McKean--Vlasov equations, and propagation of chaos.

\subsection{Our contributions}
Our primary contributions in this paper are summarised as follows:
\begin{enumerate}
  \item \Cref{thm:SILT_V_flow} proves the existence of self-intersection local times for the process $\{x(u(s),t): s\in[0,1]\}$, where $\{u(s):  s\in[0,1]\}$ is a Volterra Gaussian process (see Definition \ref{def:Volterra}) and $x$ is the random diffeomorphism satisfying equation \eqref{eq:interaction} with the initial measure $\mu_0$ being the occupation measure of the process $u.$ \Cref{thm:SILT_V_flow} is proved by means of a crucial discretisation procedure (\Cref{thm:SILT_V_flow_n}) applied to the initial occupation measure $\mu_0$ of the process $u$, a technical lemma on the determinant of the Jacobian matrix (\Cref{thm:determinant}) and an elegant ``change of variable'' formula (\Cref{thm:Weighted_SILT_V}), which allows us to describe the self-intersection local times for the process $\{x(u(s),t): s\in[0,1]\}$, in terms of the self-intersection local times of the process $u$, and an appropriate (random) weight function. 
  \item As a by-product of our analysis, we also obtain conditions on unbounded weight functions and general Gaussian processes that guarantee the existence of multiple weighted self-intersection local times (\Cref{thm:Weighted_SILT}), which is of independent interest. To the best of our knowledge, such existence results have been established only for bounded weights so far. 
  
  \item \Cref{thm:asymptotics_expectations} describes the asymptotics of self-intersection local times $\Eof{T_k^{x(u, t)}}$ for large $t$ when $x$ solves a special (sub-)class of the equation with interaction. 
  \item Drawing parallel to the deterministic literature, we establish certain stochastic exponential martingales associated with the self-intersection local times. This is inspired by a similar question related to the asymptotics of the length of deterministic smooth curves under the action of isotropic Brownian flows. If $L_t \coloneq \int^1_0\norm{\gamma^{\prime}_t(u)}\differential{u}$ is the length of the curve $\gamma_t \coloneq \phi_t\circ\gamma:[0,1]\to\mbR^d$, then it is known that the stochastic process 
$\{\exp\left(-\left(\lambda+\frac{\beta_L}{2}\right)t\right)L_t : t\ge 0\}$ is a martingale, which converges almost surely, where 
$\phi_t$ is an isotropic Brownian flow, $\lambda$ is the top Lyapunov exponent associated to the flow,  and $\beta_L$ is the characteristic constant of the flow. See \cite{dimitroff2006some} for details and precise definitions. For non-smooth random curves (such as the Volterra Gaussian process) in stochastic flows with interaction, the self-intersection local times  $T^{x(u,t)}_k$ are the geometric characteristics. In \Cref{thm:silt martingale},  we prove that $\{\exp\left((k-1)\hat{a}t-\frac{k(k-1)\hat{b}t}{2}\right)T^{x(u,t)}_k : t\ge 0\}$ is a positive, continuous square-integrable martingale for an appropriate choice of $\hat{a}$, and $\hat{b}$. We provide an explicit expression for its quadratic variation.
\end{enumerate}

We focus on  the Volterra Gaussian processes because they constitute a rich class of Gaussian processes including the Wiener process, the Brownian bridge, the fractional Brownian motion. 
On the technical side, the representation of these Gaussian processes via stochastic integrals of deterministic kernels allows us to describe the properties of the processes via the properties of kernels, which are often easier to state and verify. Note that Volterra Gaussian processes are, in general, neither Markovian nor martingales.

\subsection{Structure of the paper}
The rest of the paper is structured as follows. In \Cref{sec:SILT}, we first introduce sufficient conditions on unbounded weight functions and general Gaussian processes that guarantee the existence of multiple weighted self-intersection local times. Then, we specialise to Volterra Gaussian processes and provide sufficient conditions for the existence of multiple weighted self-intersection local times. In \Cref{sec:VGSF}, we prove the existence of multiple self-intersection local times $T^{x(u,t)}_k$ for the process $\{x(u(s),t): s\in[0,1]\}$ and 
establish the ``change of variable'' formula
$$
T^{x(u,t)}_k=T^u_k\left(\frac{1}{|\det\D x(\cdot,t)|^{k-1}}\right).
$$
In \Cref{sec:Asymptotics SILT},  we describe the asymptotics of random variable $T^{x(u,t)}_k$ for large $t$, and study properties of a related stochastic exponential martingale. 


\subsection{Notational conventions}
Let $(\Omega, \cF, \prob)$ be a probability space, large enough to carry all random elements considered in this paper. 
For $k\geqslant 2$ and $t\in [0,1]$, let
$
\Delta_k(t)\coloneq \{t_1,\ldots,t_k\in[0,1]:\ 0\leqslant t_1\ldots\leqslant t_k\leqslant t\}
$
and $\Delta_k \equiv \Delta_k(1).$ 
Let $\A=(a_{ij})^n_{i, j=1},\ \B=(b_{ij})^n_{i, j=1}$ be two matrices of dimension $n\times n.$
By $\A\odot \B=(a_{ij}b_{ij})^n_{ij=1}$, we denote the Hadamard product of matrices $\A$ and $\B.$ The Hilbert--Schmidt norm of the matrix $A$ is defined by
 $$
 \norm{A}_{HS}=\sqrt{\trace(A^{*}A)},
 $$
 where $A^{*}$ is the matrix adjoint to $A.$
Given a Hilbert space $\mbH$ and elements 
$e_1,\ldots,e_n\in \mbH$, we denote by $G(e_1,\ldots,e_n)$ the Gram determinant constructed by the elements $e_1,\ldots,e_n.$ Note that for Gaussian random variables $\xi_1,\ldots,\xi_n,$
the following relation holds
$$
\det\Cov(\xi_1,\ldots,\xi_n)=G(\xi_1,\ldots,\xi_n).
$$
We use the notation $\Var(\xi_i\mid \xi_1,\ldots,\xi_{i-1})$ for the conditional variance $\xi_i$ given $\xi_1,\ldots,\xi_{i-1}.$ Note that
$$
\det\Cov(\xi_1,\ldots,\xi_n)=\Var(\xi_1)\Var(\xi_2\mid\xi_1)\cdots \Var(\xi_n\mid\ \xi_1,\ldots,\xi_{n-1}).
$$
We use the notation $\cM$ for the set of all probability measures on $\cB(\mbR^d)$. Let us denote by $C(\mu,\nu)$ the set of all probabilities measures on $\cB(\mbR^{d}\times\mbR^d)$ with marginal projections $\mu$ and $\nu.$ We use the following notations for the set of all probability measures having finite moments of order $n,$
$$
\mathcal{M}_n=\left\{\mu\in\mathcal{M}:\  \int_{\mbR^d}\|y-z\|^n\mu(\differential{z})<\infty,\ \text{for all}\  y\in\mbR^d\right\},
$$
and for the Wasserstein distance of order $n,$ 
$$
\gamma_n(\mu,\nu)\coloneq \left(\inf_{\kappa\in C(\mu,\nu)}\int_{\mbR^d}\int_{\mbR^d}\|y-z\|^n\kappa(\differential{y},\differential{z})\right)^{\frac{1}{n}}.
$$

We will use notations $\lambda_d$ for the Lebesgue measure on $\mbR^d.$ 
For a differentiable function $\Psi:\mbR^d\to\mbR^d$ we denote the Jacobian matrix by
$\D\Psi(y)=(\partial_j\Psi_i(y))^d_{i,j=1}$, where $\Psi_i$ is the $i$-th coordinate of $\Psi.$




\section{Weighted self-intersection local times}
\label{sec:SILT}

\subsection{General $\mbR^d$-valued Gaussian processes}
\label{sec:weighted_GP}
  Let us begin by formulating a sufficient condition for the existence of $k$-multiple weighted self-intersection local times for a general $\mbR^d$-valued Gaussian process $\{u(t) \equiv (u_1(t), \ldots, u_d(t)) :  t\in[0,1]\}$. 
For the Gaussian process $u$, we assume the weight function $\rho: \mbR^d \to \mbR$ satisfies the following condition: 
\begin{align}
  \label{eq:condition_on_weight_function}  \lvert{\rho(y)}\rvert\leqslant \alpha e^{\beta\|y\|^2} \quad \text{for all } y \in \mbR^d \text{and } \alpha > 0, 0\leqslant\beta<\frac{1}{8\sup_{t\in[0,1]}\Var(u_1(t))}.
\end{align}
Now, we consider the family of approximations $\{T^u_{\varepsilon,k}(\rho) : \ve >0 \}$ for the weighted $k$-multiple self-intersection local time of the process $u$ defined in \eqref{eq:weighted_SILT_approx}. 
The following theorem, which is interesting on its own, provides a sufficient condition for the existence weighted self-intersection local times for the centred Gaussian process $u$.



\begin{theorem}
  \label{thm:Weighted_SILT} Let $\{u(t) \equiv (u_1(t),\ldots,u_d(t)):  t\in[0,1]\}$ be a centred Gaussian process in $\mathbb{R}^d$ with indepdendent and identically distributed (i.i.d.) components and $\rho:\mathbb{R}^d\to\mathbb{R}$ be a weight function satisfying \eqref{eq:condition_on_weight_function}. Let $p\ge 2$.  If
 \begin{align}
   \label{eq:Gram}
 \int_{\Delta^p_k}\frac{1}{G(u_1(t_1),\ldots,u_1(t_{pk}))^{\frac{d}{2}}}\differential{t_1}\cdots\differential{t_{pk}}<\infty,
\end{align}
then there exists a random variable $T^{u}_k(\rho)$ such that 
the family  
$\{T^u_{\varepsilon,k}(\rho) : \ve >0 \}$ converges to $T^{u}_k(\rho)$ in $L^p(\Omega,\cF,\prob)$ as $\ve\to0$. 
\end{theorem}
%
\begin{proof}[Proof of Theorem~\ref{thm:Weighted_SILT}]
  We will prove the theorem for $p=2$. The proof for $p> 2$ follows in a similar manner. To check that the family of random variables $\{T^{u}_{\varepsilon, k}(\rho):\ \ve>0\}$ is a Cauchy family in $L^2(\Omega,\cF,P)$, it suffices to check that 
\begin{align*}
\lim_{\varepsilon_1,\varepsilon_2\to0}\Eof{T^{u}_{\varepsilon_1, k}(\rho)T^{u}_{\varepsilon_2, k}(\rho)}<\infty.
  \end{align*}
  Assume that $\varepsilon_1=\varepsilon_2.$ Then,
  \begin{align}
  \label{eq:product} 
   \Eof{T^{u}_{\varepsilon, k}(\rho)^2}=\int_{\Delta^2_k}\Eof{\rho(u(t_1))\rho(u(t_{k+1}))\prod^{2k-1}_{i=1,\ i\neq k}f_{\varepsilon}(u(t_{i+1})-u(t_i))}\differential{t_1}\ldots \differential{t_{2k}}.
\end{align}
Let 
$
\tilde{u}(t_1),\ \tilde{u}(t_{k+1}),\ \tilde{u}(t_2),\ldots,\ \tilde{ u}(t_{2k})
$
 be the orthogonal system of elements  obtained from 
 $
 u(t_1),\ u(t_{k+1}),\ u(t_2),\ldots,\ u(t_{2k})
 $
 via the Gram-Schmidt orthogonalisation procedure in $L^2(\Omega,\cF,\prob)$.
 Let $a(t_i,t_j)=\frac{\Eof{\tilde{u}_1(t_i)\tilde{u}_1(t_j)}}{\Eof{\tilde{u}_1(t_j)^2}}$ and 
\begin{align*}
&{}
Q^1_{t_1 t_{k+1} t_2\ldots t_{i+1}}(y_1,y_{k+1}, y_2,\ldots,y_i)\\
&
= (a(t_{i},t_{1})-a(t_{i+1},t_1))y_1+(a(t_{i},t_{k+1})-a(t_{i+1},t_{k+1}))y_{k+1}
+\cdots+(a(t_{i},t_{i-1})-a(t_{i+1},t_{i-1}))y_{i-1}\\
&{}\quad 
+(1-a(t_{i+1},t_{i}))y_i
\end{align*}
 for any $y_1,y_{k+1},y_2,\ldots, y_i\in \mathbb{R}.$
 Let $Q_{t_1t_{k+1} t_2\ldots t_{i+1}}(y_1,y_{k+1},y_2,\ldots,y_i)$ be the vector in $\mathbb{R}^d$ with the coordinates 
 $$
 Q^1_{t_1 t_{k+1} t_2\ldots t_{i+1}}(y^1_1,y^1_{k+1},y^1_2,\ldots,y^1_i),\ldots, Q^1_{t_1 t_{k+1} t_2\ldots t_{i+1}}(y^d_1,y^d_{k+1},y^d_2,\ldots,y^d_i),
 $$ 
 where we assume that each $y_j=(y^1_j,\ldots,y^d_j)\in\mathbb{R}^d$ and   
 $
 Q_{t_1t_{k+1} t_2\ldots t_{i+1}}(y_1,y_{k+1},y_2,\ldots,y_i)=Q_{t_1t_{k+1} t_2}(y_1,y_{k+1}),
 $ for $i=1$.
 Then, the expectation in \eqref{eq:product} has the following representation
  \begin{align*}
&{}
\int_{\Delta^2_k}\Eof{\rho(\tilde{u}(t_1))\rho(\tilde{u}(t_{k+1})+a(t_{k+1},t_1)\tilde{u}(t_1))
\prod^{2k-1}_{i=1,\ i\neq k}f_{\varepsilon}(\tilde{u}(t_{i+1})-Q_{t_1 t_{k+1} t_2\ldots t_{i+1}}(\tilde{u}(t_1),\tilde{u}(t_{k+1}),\ldots,\tilde{u}(t_i)))}\differential{t_1}\cdots\differential{t_{2k}}\\
&{}=\int_{\Delta^2_k}\int_{\mathbb{R}^{2d}}\rho(y_1)\rho(y_{k+1}+a(t_{k+1},t_1)y_1)\\
&{}\quad \quad \times \int_{\mathbb{R}^{2kd-2d}}\prod^{2k-1}_{i=1,\ i\neq k}f_{\varepsilon}(y_{i+1}-Q_{t_1,t_{k+1},t_2\ldots t_{i+1}}(y_1,y_{k+1},\ldots,y_i))
\prod^{2k}_{i=1}\tilde{p}_{t_i}(y_i)\differential{y_1}\cdots \differential{y_{2k}}\differential{t_1}\cdots\differential{t_{2k}},
    \end{align*}
where $\tilde{p}_{t_i}$ is the density of the Gaussian vector $\tilde{x}(t_i),\ i=1,\ldots, 2k.$
Note that
  \begin{align*}
\int_{\mathbb{R}^{2kd-2d}}\prod^{2k-1}_{i=1,\ i\neq k}f_{\varepsilon}(y_{i+1}-Q_{t_1,t_{k+1},t_2\ldots t_{i+1}}(y_1,y_{k+1},\ldots,y_i))
\prod^{2k}_{i=2,\ i\neq k+1}\tilde{p}_{t_i}(y_i)\differential{y_2}\cdots \differential{y_{2k}} \to \tilde{p}_{t_1\ldots t_{2k}}(y_1,y_{k+1}),
    \end{align*}
as $\ve\to0,$ where
\begin{align*}
\tilde{p}_{t_1\ldots t_{2k}}(y_1,y_{k+1}) &{}=\tilde{p}_{t_2}(Q_{t_1 t_{k+1} t_2}(y_1,y_{k+1}))\tilde{p}_{t_3}(Q_{t_1 t_{k+1} t_2 t_3}(y_1,y_{k+1},Q_{t_1 t_{k+1} t_2}(y_1,y_{k+1}))
\tilde{p}_{t_{2k}}(Q_{t_1,t_{k+1},t_2\ldots t_{2k}}(y_1,y_{k+1},\ldots)).
\end{align*}
To apply Lebesgue's dominated convergence theorem, let us check that there exists a function 
 $g_{t_1,\dots,t_{2k}}:\ \mathbb{R}^{2d}\to\mathbb{R}_{+}$ such that for any $\ve>0$,
\begin{align*}
&{}\lvert\rho(y_1)\rho(y_{k+1}+a(t_{k+1},t_1)y_1)
\int_{\mathbb{R}^{2kd-2d}}\prod^{2k-1}_{i=1,\ i\neq k}f_{\varepsilon}(y_{i+1}-Q_{t_1,t_{k+1},t_2\ldots t_{i+1}}(y_1,y_{k+1},\ldots,y_i)) 
\prod^{2k}_{i=1}\tilde{p}_{t_i}(y_i)\differential{y_1}\cdots \differential{y_{2k}}\rvert
\\
&{}\quad \leqslant g_{t_1,\ldots,t_{2k}}(y_1,y_{k+1}), 
\end{align*}
for all $y_1, y_{k+1}$,    and
   $
\int_{\Delta^2_k}\int_{\mathbb{R}^{2d}}g_{t_1\ldots t_{2k}}(y_1,y_{k+1})\differential{y_{1}}\differential{y_{k+1}}\differential{t_1}\cdots\differential{t_{2k}}<\infty.
         $
Let us first integrate with respect to $y_{2k}.$ Applying Plancherel's theorem, one can see that
 \begin{align*}
 \int_{\mathbb{R}^{d}}f_{\varepsilon}(y_{2k}-Q_{t_1\ldots t_{2k}}(y_1,y_{k+1},\ldots,y_{2k-1}))\tilde{p}_{t_{2k}}(y_{2k})\differential{y_{2k}}
&{}
=\int_{\mathbb{R}^{d}}e^{i(\theta, Q_{t_1\ldots t_{2k}}(y_1,y_{k+1},\ldots,y_{2k-1}))-\frac{\ve \|\theta\|^2}{2}}e^{-\frac{\Eof{\tilde{u}_1(t_{2k})^2}\|\theta\|^2}{2}}\differential{\theta}
\\
&{}\quad 
\leqslant\frac{c}{\left(\Eof{\tilde{u}_1(t_{2k})^2}\right)^{\frac{d}{2}}}, \text{ for some } c>0.
  \end{align*}
  Repeating the same arguments $(2k-1)$ times, one can conclude that
  \begin{align*}
&{} \int_{\mathbb{R}^{2kd-2d}}\prod^{2k-1}_{i=1,\ i\neq k}f_{\varepsilon}(y_{i+1}-Q_{t_1\ldots t_{i+1}}(y_1,y_{k+1},\ldots,y_i)
\prod^{2k}_{i=2,i\neq k+1}\tilde{p}_{t_i}(y_i)\differential{y_2}\cdots \differential{y_{2k}}\leqslant \frac{c}{\prod^{2k}_{i=2,\ i\neq k+1}\left(\Eof{\tilde{u}_1(t_{i})^2}\right)^{\frac{d}{2}}}.
 \end{align*}  
 Now, applying \eqref{eq:condition_on_weight_function}, one can conclude that
  \begin{align*}
&{}\int_{\mathbb{R}^{2d}}\rho(y_1)\rho(y_{k+1}+a(t_{k+1},t_1)y_1)\tilde{p}_{t_1}(y_1)\tilde{p}_{t_{k+1}}(y_{k+1})\differential{y_1}\differential{y_{k+1}}
\\
&{}
\leqslant \alpha^2\int_{\mathbb{R}^{2d}}\myExp{(\beta+2\beta (a(t_{k+1},t_1))^2)\|y_1\|^2}\myExp{2\beta\|y_{k+1}\|^2}\tilde{p}_{t_1}(y_1)\tilde{p}_{t_{k+1}}(y_{k+1})\differential{y_1}\differential{y_{k+1}}.
   \end{align*}  
Furthermore, it follows from the condition on $\beta$ in \eqref{eq:condition_on_weight_function} that
   \begin{align*}
  \int_{\mathbb{R}^{d}}e^{2\beta\|y_{k+1}\|^2}\frac{1}{\left(2\pi\Eof{\tilde{u}_1(t_{k+1})^2}\right)^{\frac{d}{2}}}\myExp{-\frac{\|y_{k+1}\|^2}{2\Eof{\tilde{u}_1(t_{k+1})^2}}}\differential{y_{k+1}}=\frac{1}{\left(1-4\beta \Eof{\tilde{u}_1(t_{k+1})^2}\right)}<c_1, 
   \end{align*}   
   and
 \begin{align*}
 \int_{\mathbb{R}^{d}}e^{(\beta+2\beta (a(t_{k+1},t_1)^2)\|y_1\|^2}\frac{1}{\left(2\pi\Eof{\tilde{u}_1(t_{1})^2}\right)^{\frac{d}{2}}}e^{-\frac{\|y_{1}\|^2}{2\Eof{\tilde{u}_1(t_{1})^2}}}\differential{y_{1}}
=\frac{1}{1-2\beta\Eof{\tilde{u}_1(t_{1})^2}(1+2(a(t_{k+1},t_1))^2)}<c_2,
   \end{align*} 
   where $c_1,\ c_2$ are some positive constants.
   Note that
   \begin{align*} 
&{}\int_{\Delta^2_{k}}\frac{1}{\prod^{2k}_{i=2,\ i\neq k+1}\left(\Eof{\tilde{u}_1(t_{i})^2}\right)^{\frac{d}{2}}}\differential{t_1}\cdots\differential{t_k}
\leqslant c\int_{\Delta^2_k}\frac{1}{G(u_1(t_1),\ldots,u_1(t_{2k})^{\frac{d}{2}}}\differential{t_1}\cdots\differential{t_{2k}}<\infty
\end{align*}
by assumption \eqref{eq:Gram}, which completes the proof of the theorem since the case $\ve_1\neq\ve_2$ can be done similarly using the identity 
$f_{\ve_2}(z)=\left(\frac{\ve_1}{\ve_2}\right)^{\frac{d}{2}}f_{\ve_1}\left(\sqrt{\frac{\ve_1}{\ve_2}} z\right). $
\end{proof}


\subsection{Volterra Gaussian processes}
\label{sec:weighted_VGP}
Let  $\{w(t) : t\in [0,1]\}$, be a one-dimensional Wiener process.  
\begin{definition}
\label{def:Volterra} A centred Gaussian process $\{u(t) : t\in[0,1]\}$ is called a Volterra Gaussian process, if  it admits the representation
$
u(t) = 
\int^t_0k(t,s)\differential{w(s)},
$ for each $t\in(0,1]$, 
where $k\in L^2([0,1]^2)$ is a Volterra kernel, i.e., 
$k(t,s)=0$  for all $s>t$, and 
$
\sup_{t\in[0,1]}\int^t_0k(t,s)^2 \differential{s}<\infty. 
$
\end{definition}
The concept of local nondeterminism was introduced in \cite{berman1973local} as a sufficient condition that guarantees the existence of jointly continuous local times for $\mbR$-valued Gaussian processes. More precisely, if $\{u(t) : t\in[0,1]\}$ is a Gaussian process such that for any $m\geqslant 2$ and any $t_1<\ldots<t_m$, we have 
$$
\lim_{c\downarrow 0}\inf_{t_m-t_1\leqslant c}\frac{\Var(u(t_m)-u(t_{m-1})\mid u(t_1),\ldots,u(t_{m-1})}{\Var(u(t_m)-u(t_{m-1})}>0,
$$
then 
the process $u$ is said to be a locally nondeterministic. If the kernel $k$ of a Volterra Gaussian process $u$ satisfies 
$$
 \lim_{c\downarrow 0}\inf_{0<t-s\leqslant c}\frac{\int^t_sk^2(t,r)\differential{r}}{\int^s_0(k(t,r)-k(s,r))^2 \differential{r}}>0,
$$
  then $u$ is locally nondereministic.
 The following extension of the local nondeterminism condition for Volterra Gaussian processes was introduced in \cite{harang2022regularity}.
 \begin{definition}
 \label{def:LND}
  If the kernel $k$ of a Volterra Gaussian process $u$ satisfies 
   \begin{align}
\label{eq:LND} 
 \inf_{t\in[0,1]}\inf_{s\in[0,t]}\frac{1}{(t-s)^{\zeta}}\int^t_sk^2(t,r)\differential{r}>0
  \end{align}
  for some $\zeta>0,$ then the process $u$ is said to be $(2,\zeta)$-locally nondeterministic.
  \end{definition} 

We refer the readers to \cite{Izyumtseva2026VGP} for an extensive discussion on self-intersection local times of Volterra Gaussian processes. 
Let us now describe the class of Volterra Gaussian processes for which an analogue of \Cref{thm:Weighted_SILT} holds.
Let $w_1,\ldots, w_d$ be independent one-dimensional Wiener processes.
Consider the $\mbR^d$-valued Volterra Gaussian process 
$$
u(t)=(u_1(t),\ldots, u_d(t)),\ t\in [0,1],
$$ 
where  the coordinate $u_i$ is an $\mbR$-valued Volterra Gaussian process generated by the kernel $k$ and the Wiener process $w_i$ for  each $i=1,\ldots,d$. 
Assume that there is a constant $L>0$ such that for any $t,s,r\in[0,1],$
  \begin{align}
\label{eq:Lip} 
|k(t,r)-k(s,r)|\leqslant L|t-s|.
 \end{align}
Using assumption \eqref{eq:Lip}, one can see that 
 $ \Eof{\|u(t)-u(s)\|^2}\leqslant c|t-s|$ for any $s,\ t\in[0,1]$, and, in general, for any $p\geqslant 2$
 $ \Eof{\|u(t)-u(s)\|^{2p}}\leqslant c_1|t-s|^{p},$
  where $c,\ c_1$ are positive constants. Hence, by Kolmogorov's continuity criterion, $u$ has a continuous modification. Henceforth, we will work with this continuous modification. The next theorem describes conditions on weight-functions and $\zeta$ that guarantee the existence of multiple weighted self-intersection local times for Volterra Gaussian processes.

    \begin{theorem}
\label{thm:Weighted_SILT_V} Assume that a Volterra Gaussian process $\{ u(t): t\in[0,1]\}$ in $\mathbb{R}^d$ is $(2,\zeta)$-locally nondeterministic for $\zeta\in (0,\frac{2}{d})$ and
the weight function $\rho:\mathbb{R}^d\to\mathbb{R}$  satisfies  \eqref{eq:condition_on_weight_function}. 
Then, there exists a random variable $T^{u}_k(\rho)$ such that 
the family  
$\{T^u_{\varepsilon,k}(\rho) : \ve >0 \}$ converges to $T^{u}_k(\rho)$ in $L^p(\Omega,\cF,\prob)$ as $\ve\to0$, for any $p\ge 2.$
\end{theorem}
\begin{proof}[Proof of Theorem~\ref{thm:Weighted_SILT_V}] 
 Note that 
 \begin{align*} 
 G(u_1(t_1),\ldots,u_1(t_{pk}))&{}=\Var(u_1(t_1))\prod^{pk}_{i=2}\Var(u_1(t_i)\mid u_1(t_1),\ldots,u_1(t_{i-1}))\\
 &{}\geqslant \Var(u_1(t_1))\prod^{pk}_{i=2}\Var(u_1(t_i)\mid w_1(t_1),\ldots,w_1(t_{i-1}))\\
 &{}=\int^{t_1}_0k^2(t_1,r)\differential{r}\prod^{pk}_{i=2}\int^{t_i}_{t_{i-1}}k^2(t_i,r)\differential{r}
 \geqslant c t^{\zeta}_1\prod^{pk-1}_{i=1} (t_{i+1}-t_i)^\zeta,
 \end{align*}
 where $c$ is some positive constant by virtue of the $(2, \zeta)$-local nondeterminism property.
 Hence, 
  \begin{align}
    \int_{\Delta_{pk}}\frac{1}{G(u_1(t_1),\ldots,u_1(t_{pk})^{\frac{d}{2}}}\differential{t_1}\cdots\differential{t_{pk}}
           \leqslant c_1 \int_{\Delta_{pk}}\frac{1}{t^{\frac{\zeta d}{2}}_1\prod^{pk-1}_{i=1} (t_{i+1}-t_i)^{\frac{\zeta d}{2}}}\differential{t_1}\cdots\differential{t_{pk}} < \infty, 
           \label{eq:gram}
    \end{align}
    since for $\zeta\in(0,\frac{2}{d})$. The theorem now follows from Theorem~\ref{thm:Weighted_SILT}. 

 \end{proof}
    \section{Volterra Gaussian processes in stochastic flows}
\label{sec:VGSF}
\subsection{Stochastic flows with interaction}
\label{sec:stoch_flows}
Let $W$ be a Brownian sheet defined on the Borel subsets of $[0,\infty)\times\mathbb{R}^d$ with the finite Lebesgue measure (see \cite[Example 3.5]{Lifshits2012LecturesOnGaussianProcesses}).
Consider the functions $a:\ \mbR^d\times\mathcal{M}_m\to\mbR^d$ and $b:\ \mbR^d\times\mathcal{M}_m\times \mbR^d \to \mbR^d$ such that for all for all $x \in \mbR^d$ and $\nu \in \mathcal{M}_m$,  $b(x, \nu, \cdot) \in L^2(\mbR^d,\mbR^d)$, the equivalence class of square-integrable $\mbR^d$-valued functions with respect to the Lebesgue measure on $\mbR^d$.  Let us equip the space $\mbR^d \times \cM_m$ with the metric $\dist_m$ defined as follows:
\begin{align}
  \dist_m((x,\nu), (y,\mu)) = \norm{x-y} + \gamma_m(\nu,\mu),
\end{align}
where $\gamma_m$ is the $m$-th order Wasserstein metric on $\cM_m$. Let $\mu_0\in\cM_m$ be possibly random. Let us define the stochastic differential equation with interaction introduced by \cite{dorogovtsev2023measure}.
\begin{definition}
\label{def:Eq_interaction}
The stochastic differential equation
\begin{align}
\label{eq:Eq_interaction}
\begin{cases}
\differential{x}(v,t)=a(x(v,t),\mu_t)\differential{t}+\int_{\mbR^d}b(x(v,t),\mu_t,z)W(\differential{t},\differential{z}), \\	
x(v,0)=v,\ v\in\mbR^d, \\
\mu_t=\mu_0\circ x(\cdot,t)^{-1},\ t\geqslant 0, 
\end{cases}
\end{align}
is said to be the equation with interaction.
\end{definition}
Here $\{x(v,t) : t\geqslant 0\}$ is the trajectory of a particle starting from a point $v\in\mbR^d,$ and the probability measure $\mu_0$ describes the initial mass distribution of the particles. The measure $\mu_t$, which is the push-forward of $\mu_0$ by the mapping $x(\cdot,t):\mbR^d\to\mbR^d$, describes the mass distribution of particles at time $t.$ One of the interesting features of the equation with interaction is that the equation for the trajectory of the particle starting from a point $v\in\mbR^d$ contains  information about particles starting from all other points through $\mu_t$ in the coefficients. Let $\cF_t=\sigma\{\mu_0, W(\Delta) : \Delta\in \mathcal{A}_t\}$
where $\cA_t=\{A\in\cB([0,t]\times\mbR^d):\ \lambda_1\otimes\lambda_d(A)<\infty\}$. We include all $\prob$-null sets in  $\cF_t$ for all $t\geq0$ so that the filtered probability space $(\Omega, \cF, (\cF_t)_{t\ge 0}, \prob)$ is complete.
\begin{definition}
\label{def:solution}
The solution to the Cauchy problem \eqref{eq:Eq_interaction} corresponding to the coefficients $a,\ b$ and initial measure $\mu_0\in\cM_m$ is an  $\mbR^d$-valued random field $\{x(v,t): v\in\mbR^d, t\in[0,\infty)\}$ such that:
\begin{enumerate}
\item
For every $t\geq0$, the restriction of $x$ to the interval $[0,t]$ is $\cB(\mbR^d)\times\cB([0,t])\times\cF_t$-measurable. 
\item
For all $v\in\mbR^d$ and $t\geqslant 0$, the integral form of \eqref{eq:Eq_interaction} holds almost surely. 
\end{enumerate}
\end{definition}
The following two technical results, borrowed from \cite{dorogovtsev2023measure}, provide sufficient conditions  for the existence of a unique solution to \eqref{eq:Eq_interaction}, and provide moment estimates for the solution.  
\begin{theorem}
\label{thm:existence_solution} Suppose that 
there exists a constant $L>0$ such that for any $v_1,v_2\in\mbR^d,\ \mu^{(1)},\mu^{(2)}\in \cM_m$, 
$$
\|a(v_2,\mu^{(2)})-a(v_1,\mu^{(1)})\|+\Big(\int_{\mbR^d}\|b(v_2,\mu^{(2)},z)-b(v_1,\mu^{(1)},z)\|^2\differential{z}\Big)^{\frac{1}{2}} \leq L \dist_m((v_1, \mu^{(1)}), (v_2, \mu^{(2)})) .
$$
Moreover, suppose that $b$ is continuous with respect to variables $v$ and $\mu$. Then, the stochastic differential equation \eqref{eq:Eq_interaction} has a solution, which is unique, and for every $t>0$, the measure $\mu_t$ is a random element in $\cM_m.$
\end{theorem}
\begin{lemma}
\label{lem:moments}
Let $\mu^{(1)},\mu^{(2)}\in\cM_m$ be deterministic and let $x_1,x_2$ be the corresponding solutions to \eqref{eq:Eq_interaction}. Then, for every $T>0$, there exists a constant $c>0$ such that
$$
\Eof{\sup_{t\in [0,T]}\gamma_m(\mu^{(1)}_t,\mu^{(2)}_t)} \leqslant\ c\gamma_m(\mu^{(1)},\mu^{(2)})^m,\quad \text{ and } \quad 
\Eof{\sup_{t\in [0,T]}\norm{x_1(v,t)-x_2(v,t)}^m }\leqslant c\gamma_m(\mu^{(1)},\mu^{(2)})^m,\ v\in\mbR^d.
$$
\end{lemma}
The proofs of \Cref{thm:existence_solution} and \Cref{lem:moments} can be found in \cite[p. 49 - 51, Lemma 2.4.1, Theorem 2.4.1]{dorogovtsev2023measure}.
The aim of this section is to prove the existence of self-intersection local times for Volterra Gaussian processes in stochastic flows with interaction. Let $\{u(t) : t\in[0,1]\}$ be a Volterra Gaussian process in $\mbR^d$, independent of $W.$ Assume that $u$ is $(2,\zeta)$-locally nondeterministic for some $\zeta\in (0,\frac{2}{d})$. Let $\mu$ be the occupation measure of $u.$ It was proved in \cite{harang2022regularity} that $u$ has a jointly continuous local time $l$. Consider the equation with interaction \eqref{eq:Eq_interaction} with $\mu_0=\mu$. If coefficients satisfy the  conditions of Theorem \ref{thm:existence_solution}, then there exists a unique solution $\{x(v,t):\ v\in\mbR^d,\ t\geqslant 0\}$ and  $\mu_t$ is the occupation measure of the process $\{x(u(s),t):\ s\in[0,1]\}.$ If $x(\cdot,t):\mbR^d\to\mbR^d$ is diffeomorphism, then for any bounded and measurable function $\phi:\mbR^d\to\mbR^d$,
\begin{align*}
&{}\int_{\mbR^d}\phi(y)\mu_t(\differential{y})=\int_{\mbR^d}\phi(x(y,t))\mu(\differential{y}),
\text{ and } \int_{\mbR^d}\phi(x(y,t)l(y)\differential{y}=\int_{\mbR^d}\phi(y)l(x^{-1}(y,t))\frac{1}{\lvert \det \D x(y,t)\rvert}\differential{y}.
\end{align*}
So,  we can immediately conclude that $\mu_t\ll\lambda_d$ for all $t\ge 0$. Hence, for each $t>0,$ the local time $l_t$ of the process $\{x(u(s),t) :  s\in[0,1]\}$ exists and it has the following representation
$$
l_t(y)=l(x^{-1}(y,t))\frac{1}{\lvert \det \D x(y,t)\rvert},\ y\in\mbR^d,
$$
where $\D x(y,t)=(\partial_j x_i(y,t))^d_{i,j=1}$ is the Jacobian matrix.  
Unfortunately, the same method cannot be applied to determine self-intersection local times.
That is why we will use approximation arguments, as done before,  to define the random variable $T^{x(u,t)}_k.$ 
Nevertheless, the Jacobian matrix and its determinant will play a crucial role in our analysis. The following theorem describes the dynamics of the Jacobian matrix.

\begin{theorem}
  Assume that the coefficients $a$ and $b$ of \eqref{eq:Eq_interaction} are continuously differentiable with respect to $v$ for all $\mu\in\cM_m$. Moreover, assume there exist constants $C, L>0$ such that for all $v_1,v_2\in\mbR^d,\ \mu^{(1)},\mu^{(2)}\in\cM_m, $
\begin{align*}
\norm{\D a(v_2,\mu^{(2)})-\D a(v_1,\mu^{(1)})}+\left(\int_{\mbR^d}\|\D b(v_2,\mu^{(2)},z)-\D b(v_1,\mu^{(1)},z)\|^2_{HS}\differential{z}\right)^{\frac{1}{2}} &{}\leq L \dist_m((v_1, \mu^{(1)}), (v_2, \mu^{(2)})), 
\\
\sup_{\mu\in\cM_m}\sup_{v\in\mbR^d}\left(\norm{\D a(v,\mu)}+\Big(\int_{\mbR^d}\|\D b(v,\mu,z)\|^2_{HS}\differential{z}\right)^{\frac{1}{2}}\Big)&{} \leqslant C.
\end{align*}
Then, there exists a modification of the solution to \eqref{eq:Eq_interaction} that is a stochastic flow of diffeomorphisms (see Appendix \ref{def:stochastic flows}) such that $\det\D x(v,t)$, the determinant of the Jacobian,  satisfies the following equation 
\begin{align*}
\det\D x(v,t) &{} =\exp\left(\int^t_0 \trace\  \D a(x(v,s),\mu_s)\differential{s}
-\frac{1}{2} \trace\int^t_0\int_{\mbR^d}\D b(x(v,s),\mu_s,z)\odot \D b(x(v,s),\mu_s,z)\differential{z}\differential{s}\right.\\
&{}\quad\quad\quad  \left.+\trace\int^t_0\int_{\mbR^d}\D b(x(v,s),\mu_s,z)W(\differential{s},\differential{z})\right).
\end{align*}
\label{thm:determinant}
\end{theorem}

\begin{proof}[Proof of \Cref{thm:determinant}]
The fact that there exists a modification of the solution to \eqref{eq:Eq_interaction} that is a stochastic flow of diffeomorphisms follows from \cite{dorogovtsev2023measure}. 
The Jacobian matrix $\D x(v,t),\ v\in\mbR^d,\ t\geqslant 0$ satisfies the following stochastic differential equation 
\begin{align}
\label{eq:derrivative}
\begin{cases}
\differential{\D x}(v,t)=\D a(x(v,t),\mu_t)\D x(v,t)\differential{t}+\int_{\mbR^d}\D b(x(v,t),\mu_t,z)W(\differential{t},\differential{z})\D x(v,t),\\	
\D x(v,0)=I_{d},\ v\in\mbR^d,\\
\mu_t=\mu_0\circ x(\cdot,t)^{-1},\ t\geqslant 0,
\end{cases}
\end{align}
where $I_{d}$ is the $d\times d$ identity matrix. Choosing an orthonormal basis, the stochastic differential equation \eqref{eq:derrivative} can be equivalently rewritten as 
\begin{align}
\label{eq:derrivative_equivqlent}
\begin{cases}
\differential{Dx}(v,t)=\D a(x(v,t),\mu_t)\D x(v,t)\differential{t}+\sum^{\infty}_{k=1}\D b_k(x(v,t),\mu_t)\differential{w}_k(t)\D x(v,t),\\	
Dx(v,0)=I_{d},\ v\in\mbR^d,\\
\mu_t=\mu_0\circ x(\cdot,t)^{-1},\ t\geqslant 0,
\end{cases}
\end{align}
where $w_k(t),\ t\in[0,1],\ k\geqslant 1$ are independent $\mbR^d$-valued Wiener processes, and $b_k,\ k \geq 1$ are appropriate square-integrable functions. 
It follows from the It\^o formula that
\begin{align*}
&{} \differential{\det\D x}(v,t)=\sum^d_{i,j=1}\frac{\partial \det\D x(v,t)}{\partial x_{ij}}\differential{Dx(v,t)_{ij}}
+\frac{1}{2}\sum^d_{i,j,l,m=1}\frac{\partial^2 \det \D x(v,t)}{\partial x_{ij}\partial x_{l,m}}\differential{\langle \D x(v,\cdot)_{ij},\D x(v,\cdot)_{lm}\rangle_t}.
\end{align*}
%
Note that 
\begin{align*}
  \langle \D x(v,\cdot)_{ij},\D x(v,\cdot)_{lm}\rangle_t & = \begin{cases}
    \sum^{\infty}_{k=1}(\D b_k(x(v,t),\mu_t)\D x(v,t))_{ij}(\D b_k(x(v,t),\mu_t)\D x(v,t))_{lm} \differential{t} & \text{ if } i=l, \\
    0 & \text{ if } i\neq l.
  \end{cases}
\end{align*}
%
Then, 
\begin{align*}
&{}\sum^d_{i,j,l,m=1}\frac{\partial^2 \det\D x(v,t)}{\partial x_{ij}\partial x_{lm}}\differential{\langle \D x(v,\cdot)_{ij},\D x(v,\cdot)_{im}\rangle_t}\\
&{}=\sum^d_{i,j,m=1}\sum^{\infty}_{k=1}\frac{\partial^2 \det\D x(v,t)}{\partial x_{ij}\partial x_{im}}(\D b_k(x(v,t),\mu_t)\D x(v,t))_{ij}(\D b_k(x(v,t),\mu_t)\D x(v,t))_{im}\differential{t}=0.
\end{align*}
The last equality follows from the fact that $\frac{\partial^2 \det A}{\partial a_{ij}\partial a_{im}}
 = (-1)^{1+m}\frac{\partial \det M_{im}}{\partial a_{ij}}=0,
$
for an invertible $A\in\mbR^{d\times d},$ 
where $M_{im}$ is the minor of the matrix that remains after deleting the $i$-th row and $m$-th
column. Hence, 
\begin{align*}
\differential{\det \D x}(v,t)&{}=\sum^d_{i,j=1}\frac{\partial \det \D x(v,t)}{\partial x_{ij}}\D a(x(v,t),\mu_t)\D x(v,t)\differential{t}
+\sum^d_{i,j=1}\sum^{\infty}_{k=1}\frac{\partial \det \D x(v,t)}{\partial x_{ij}}\D b_k(x(v,t),\mu_t)\differential{w}_k(t)\D x(v,t)\\
&{}= \trace\ \D a(x(v,t),\mu_t)\det\D x(v,t)\differential{t}
+\sum^{\infty}_{k=1}\trace\ \D b_k(x(v,t),\mu_t)\differential{w}_k(t)\det\D x(v,t),
\end{align*}
where to get the last equality we apply the following identity: Let $A$ be invertible, then
$$
\sum^d_{i,j=1}\frac{\partial\det A}{\partial a_{ij}}(BA)_{ij}=\det A\ \trace\  B.
$$
Using the It\^o formula, we  deduce that the solution to the stochastic differential equation
\begin{align*}
\differential{\det\D x}(v,t)&{} = \trace\ \D a(x(v,t),\mu_t)\det\D x(v,t)\differential{t}
+\sum^{\infty}_{k=1}\trace\ \D b_k(x(v,t),\mu_t)\differential{w}_k(t)\det\D x(v,t)
\end{align*}
has the following representation
\begin{align*}
\differential{\det\D x}(v,t)&{} =\exp\left(\int^t_0 (\trace\  \D a(x(v,s),\mu_s)ds\right.
\left.-\frac{1}{2}\sum^{\infty}_{k=1}\trace\int^t_0\D b_k(x(v,s),\mu_s)\odot \D b_k(x(v,s),\mu_s))\differential{s}\right.\\
&{}\quad \quad \quad \quad 
\left.+\sum^{\infty}_{k=1}\int^t_0 \trace\  \D b_k(x(v,s),\mu_s)\differential{w_k}(s)\right),
\end{align*}
which finishes the proof.
\end{proof}

To construct weighted self-intersection local times for Volterra Gaussian process, define the weight function 
$$
\rho(v,t)=\frac{1}{\mid \det \D x(v,t)\mid},\ v\in\mbR^d, t\ge 0.
$$
The next lemma shows that the moments of $\rho$ are uniformly bounded.  We omit the proof of \Cref{lem:moment estimate} since it can be adapted from  \cite[Lemma 3.2]{dorogovtsev2023intermittency}.

\begin{lemma}
\label{lem:moment estimate} Assume that the coefficients $a$ and $b$ of \eqref{eq:Eq_interaction} satisfy the conditions of \Cref{thm:existence_solution} and \Cref{thm:determinant}. Then, for all $T>0$ and $k\geqslant 1$,
\begin{align}
&{}  
 \sup_{v\in\mbR^d}\Eof{\sup_{0\leqslant t\leqslant T}\mid\rho(v,t)\mid^k}<\infty.
 \label{eq:moments estimate}
\end{align}
\end{lemma}


\begin{theorem}
\label{thm:detrminant difference estimate} 
Let $\mu^{(1)}$ and $\mu^{(2)}$ be deterministic probability measures in $\cM_m$ and $x_1$ and $x_2$ be the corresponding solutions to \eqref{eq:Eq_interaction} with the initial measures $\mu^{(1)}$ and $\mu^{(2)}$, respectively.  
Let
$\rho^{(k)}_i(v,t)=\frac{1}{\mid \det \D x_i(v,t)\mid^{k-1}},$ for $ i=1,2.$   Then, 
$$
\sup_{v\in\mbR^d}\sup_{t\in[0,T]}\Eof{\left(\rho^{(k)}_2(v,t)-\rho^{(k)}_1(v,t)\right)^2}\leqslant c\gamma_4(\mu^{(1)},\mu^{(2)})^2,
$$
for all $T>0$, $k\geqslant 2$,  and $c$ is some a constant.
\end{theorem}
\begin{proof}[Proof of \Cref{thm:detrminant difference estimate}]
 Note that $\rho^{(k)}_i$ satisfies the following stochastic differential equation
\begin{align*}
\differential{\rho}^{(k)}_i(v,t) &{}=-(k-1)\trace\  \D a(x_i(v,t),\mu^{(i)}_t)\rho^{(k)}_i(v,t)\differential{t} 
\nonumber \\
&{}\quad 
+\frac{k(k-1)}{2}\int_{\mbR^d}\trace\ \D b(x_i(v,t),\mu^{(i)}_t,z)\odot \D b(x_i(v,t),\mu^{(i)}_t,z))\differential{z}\ \rho^{(k)}_i(v,t)\differential{t} \nonumber \\
&{}\quad 
-(k-1)\trace\int_{\mbR^d}\D b(x_i(v,t),\mu^{(i)}_t,z)\rho^{(k)}_i(v,t)W(\differential{t},\differential{z}). 
\end{align*}
Therefore, for each $v\in\mbR^d$ and $t\in[0,T]$, 
 \begin{align*}
\Eof{\left(\rho^{(k)}_2(v,t)-\rho^{(k)}_1(v,t)\right)^2} &{} 
\leqslant c \Eof{\left(\int^t_0\left(\trace\  \D a(x_2(v,s),\mu^{(2)}_s)-\trace\  \D a(x_1(v,s),\mu^{(1)}_s)\right)\rho^{(k)}_2(v,s)\differential{s}\right)^2} \nonumber\\
&{}\quad 
+c \Eof{\left(\int^t_0 \trace\ \D a(x_1(v,s),\mu^{(1)}_s)(\rho^{(k)}_2(v,s)-\rho^{(k)}_1(v,s))\differential{s}\right)^2 }\nonumber\\
&{}\quad 
+\frac{ck^2}{4}\E\left(\left(\int^t_0\int_{\mbR^d}\left(\trace\ \D b(x_2(v,s),\mu^{(2)}_s,z)\odot \D b(x_2(v,s),\mu^{(2)}_s,z)\right.\right. \right. \nonumber\\
&{}\quad 
\left.\left. \left.-\trace\ \D b(x_1(v,s),\mu^{(1)}_s,z)\odot \D b(x_1(v,s),\mu^{(1)}_s,z)\right)\rho^{(k)}_2(v,s)\differential{z}\differential{s}\right)^2 \right)\nonumber\\
&{}\quad 
+\frac{ck^2}{4}\E\left(\left(\int^t_0\int_{\mbR^d}\trace\ \D b(x_1(v,s),\mu^{(1)}_s,z)\odot \D b(x_1(v,s),\mu^{(1)}_s,z)\right. \right.\nonumber\\
&{}\quad \quad \quad 
\left.\left. \times (\rho^{(k)}_2(v,s)-\rho^{(k)}_1(v,s))\differential{z}\differential{s}\right)^2 \right)\nonumber\\
&{}\quad 
+c\E\left(\left(\int^t_0\int_{\mbR^d}(\trace\  \D b(x_2(v,s),\mu^{(2)}_s,z)-\trace\  \D b(x_1(v,s),\mu^{(1)}_s,z))\right. \right. \nonumber\\
&{}\quad \quad \quad 
\left.\left. \times\rho^{(k)}_2(v,s)W(\differential{s},\differential{z})\right)^2 \right)\nonumber\\
&{}\quad 
+c\Eof{\left(\int^t_0\int_{\mbR^d}\trace\  \D b(x_1(v,s),\mu^{(1)}_s,z))(\rho^{(k)}_2(v,s)-\rho^{(k)}_1(v,s))W(\differential{s},\differential{z})\right)^2},  
 \end{align*}
 where $c= 25(k-1)^2$. 
Applying \Cref{thm:determinant} and Lemma \ref{lem:moments} to each summand of the right-hand side of the inequality, one can conclude that
 \begin{align*}
\Eof{\left(\rho^{(k)}_2(v,t)-\rho^{(k)}_1(v,t)\right)^2}
&{}\leqslant c_1\gamma_4(\mu^{(1)},\mu^{(2)})^2+c_2\int^t_0\Eof{\left(\rho^{(k)}_2(v,s)-\rho^{(k)}_1(v,s)\right)^2}\differential{s},
\end{align*}
for some positive constants $c_1,c_2.$ 
Applying Gr\"{o}nwall's inequality, we finish the proof of the theorem. 
\end{proof}

\subsection{Self-intersection local times}
Let $\{u(t),\ t\in[0,1]\}$ be a $(2,\zeta)$ locally nodeterministic Volterra Gaussian process in $\mathbb{R}^d$ with $\zeta\in(0,\frac{2}{d}).$ Let $\mu$ be the occupation measure of the process $u$. Note that $\mu$ is a random element of $\cM_m$ for any $m\geqslant 1.$ Consider the stochastic differential equation with interaction \eqref{eq:Eq_interaction} with the initial measure $\mu_0=\mu$ and assume that $u$ and $W$ are independent. Suppose that coefficients of \eqref{eq:Eq_interaction} satisfy the conditions of  \Cref{thm:determinant}, then there exists a version of $x$ such that \eqref{eq:Eq_interaction} is a stochastic flow of diffeomorphisms and $\mu_t$ is also a random element of $\cM_m$ for any $m\geqslant 1$ (see \cite[Lemma 1.2.4]{dorogovtsev2023measure}). Moreover, $\mu_t$ is the occupation measure of the stochastic process $\{x(u(s),t),\ s\in[0,1]\}.$ The main aim of this section is to prove the existence of random variable
$$
T^{x(u,t)}_k=\int_{\Delta_k}\prod^{k-1}_{i=1}\delta_0(x(u(s_{i+1},t))-x(u(s_{i},t)))\,\differential{s}_1\cdots \differential{s}_k
$$
for each $k\geqslant 2$ and $t>0.$ 
Consider the following approximations for $T^{x(u,t)}_k$
\begin{align*}
T^{x(u,t)}_{\ve,k}
=\int_{\Delta_k}\frac{1}{\mid \det\ \D x(u(s_1),t)\mid^{k-1}}\prod^{k-1}_{i=1}f_{\ve}(u(s_{i+1})-u(s_{i}))\,\differential{s}_1\cdots \differential{s}_k.
\end{align*}
To prove the existence of weighted self-intersection local times for the process
$\{x(u(s),t),\ s\in[0,1]\}$, we first approximate the occupation measure of the process $u$ by a sequence of discrete measures defined as follows
$$
\mu^{(n)}(\cdot)=\frac{1}{n}\sum^n_{k=1}\delta_{u(\frac{k}{n})}(\cdot).
$$
The measures $\mu$ and $\mu^{(n)}$ are random elements of $\cM_m$. 
 \begin{lemma}
\label{lem:approximations of occup measure} 
The sequence of random measures $\{\mu^{(n)} :  n\ge 1\}$ converges to $\mu$ in the metric space $(\cM_m, \gamma_m)$  as $n \to \infty$  almost surely and in $L^m(\Omega, \cF, \prob)$, for all $m\ge 1$. 
\end{lemma}
\begin{proof}[Proof of \Cref{lem:approximations of occup measure}]
The proof follows from the estimate
$
 \gamma_m(\mu,\mu^{(n)})^m\leqslant\sum^{n-1}_{k=0} \int^{\frac{k+1}{n}}_{\frac{k}{n}}\norm{u\left(\frac{k}{n}\right)-u(s)}^m\differential{s}
$
and the continuity of the process $u.$
\end{proof}
Let $x_n$ be the solution to  the equation with interaction \eqref{eq:Eq_interaction} with the initial measure $\mu_0=\mu^{(n)}$. Then $\mu^{(n)}_t (\cdot)=\frac{1}{n}\sum^n_{k=1}\delta_{x_n(u(\frac{k}{n}),t)}(\cdot).$ Suppose that the coefficients of \eqref{eq:Eq_interaction} satisfy the conditions of \Cref{thm:determinant}, then there exists a version of $x_n$ such that \eqref{eq:Eq_interaction} is a stochastic flow of diffeomorphisms. Let us consider the existence of random variable
$$
T^{x_n(u,t)}_k=\int_{\Delta_k}\prod^{k-1}_{i=1}\delta_0(x_n(u(s_{i+1},t))-x_n(u(s_{i},t)))\differential{s}_1\cdots \differential{s}_k
$$
for each $k\geqslant 2$ and $t>0.$ Consider the approximations
$$
T^{x_n(u,t)}_{\ve,k}=\int_{\Delta_k}\frac{1}{|\det\ \D x_n(u(s_1),t)|^{k-1}}\prod^{k-1}_{i=1}f_{\ve}(u(s_{i+1})-u(s_{i}))\differential{s}_1\cdots \differential{s}_k.
$$

  \begin{theorem}
\label{thm:SILT_V_flow_n} Let $\{u(t),\ t\in[0,1]\}$ be a $(2,\zeta)$-locally nondeterministic Volterra Gaussian process in $\mbR^d$ with $\zeta\in(0,\frac{d}{2})$.
Then, the collection of random variables $\{T^{x_n(u,t)}_{\ve,k} : \ve >0\}$ converges in $L^2(\Omega, \cF, \prob)$ to the random variable $T^{x_n(u,t)}_k$ as $\ve \to 0$, and the following formula holds: 
$$
T^{x_n(u,t)}_k=T^u_k\left(\frac{1}{\lvert \det \D x_n \rvert^{k-1}}\right),\quad 
 \text{for all}\  k\geqslant 2.
$$
\end{theorem}
\begin{proof}[Proof of Theorem~\ref{thm:SILT_V_flow_n}] To prove the theorem we will apply the same strategy used in the proof of Theorem \ref{thm:Weighted_SILT}. The main difference now is  that the weight function is random and depends on a finite number of values of the process $u.$ For readers' convenience we repeat the main steps of proof of Theorem \ref{thm:Weighted_SILT} highlighting new features related to random weights.
To prove the statement, we  check that
$
\lim_{\ve\to0}\Eof{\left(T^{x_n(u,t)}_{\ve,k}\right)^2}<\infty.
$
Let 
$
\tilde{u}\left(\frac{1}{n}\right),\ldots,\tilde{u}(1),\tilde{u}(s_1),\ \tilde{u}(s_{k+1}),\ \tilde{u}(s_2),\ldots,\ \tilde{ u}(s_{2k})
$
 be the orthogonal system of elements in $L^2(\Omega,\cF,\prob)$ obtained from 
 $
 u\left(\frac{1}{n}\right),\ldots,\ u(1),\ u(s_1),\ u(s_{k+1}),\ u(s_2),\ldots,\ u(s_{2k})
 $
 via the Gram--Schmidt orthogonalisation procedure. 
\begin{align*}
\Eof{\left(T^{x_n(u,t)}_{\ve,k}\right)^2}&{}=\int_{\Delta^2_k}\E_W \E_{\tilde{u}\left(\frac{1}{n}\right),\ldots,\tilde{u}(1)} \E_{\tilde{u}(s_1), \tilde{u}(s_2), \ldots, \tilde{u}(s_{2k})}\frac{1}{|\det\ \D x_n(\tilde{u}(s_1)+a_{s_1}(\tilde{u}\left(\frac{1}{n}\right),\ldots,\tilde{u}(1)),t)|^{k-1}}\\
&{}\quad\quad  \times\frac{1}{|\det\ \D x_n(\tilde{u}(s_{k+1})+a_{s_{k+1}}(\tilde{u}\left(\frac{1}{n}\right),\ldots,\tilde{u}(1))+a(s_{k+1},s_1)\tilde{u}(s_1),t)|^{k-1}}\\
&{}\quad \quad \times\prod^{2k-1}_{i=1,i\neq k}f_{\ve}\left(\tilde{u}(s_{i+1})-Q_{\frac{1}{n},\ldots,s_{i+1}}\left(\tilde{u}\left(\frac{1}{n}\right),\ldots,\tilde{u}(s_{i})\right)\right)\differential{s_1}\cdots \differential{s_{2k}}\\
&{}=\int_{\Delta^2_k}\E_W \E_{\tilde{u}\left(\frac{1}{n}\right),\ldots,\tilde{u}(1)}\int_{\mbR^{2kd}}\frac{1}{|\det\ \D x_n(y_1+a_{s_1}(\tilde{u}\left(\frac{1}{n}\right),\ldots,\tilde{u}(1)),t)|^{k-1}}\\
&{}\quad\quad  \times\frac{1}{|\det\ \D x_n(y_{k+1}+a_{s_{k+1}}(\tilde{u}\left(\frac{1}{n}\right),\ldots,\tilde{u}(1))+a(s_{k+1},s_1)y_1,t)|^{k-1}}\\
&{}\quad\quad  \times\prod^{2k-1}_{i=1,i\neq k}f_{\ve}\left(y_{i+1}-Q_{\frac{1}{n},\ldots,s_{i+1}}\left(\tilde{u}\left(\frac{1}{n}\right),\ldots,y_{i}\right)\right)
\prod^{2k}_{i=1}\tilde{p}_{s_i}(y_i)\differential{y_1}\cdots \differential{y_{2k}}\differential{s_1}\cdots\differential{s_{2k}},
\end{align*}
where 
$
a_{s_j}\left(\tilde{u}\left(\frac{1}{n}\right),\ldots,\tilde{u}(1)\right)=\sum^{n}_{i=1}a\left(s_j,\frac{i}{n}\right)\tilde{u}\left(\frac{i}{n}\right).
$
Note that
  \begin{align*}
&{}\int_{\mathbb{R}^{2kd-2d}}\prod^{2k-1}_{i=1,i\neq k}f_{\ve}\left(y_{i+1}-Q_{\frac{1}{n},\ldots,s_{i+1}}\left(\tilde{u}\left(\frac{1}{n}\right),\ldots,y_{i}\right)\right)
\prod^{2k}_{i=2,\ i\neq k+1}\tilde{p}_{s_i}(y_i)\differential{y_2}\cdots \differential{y_{2k}}\to \tilde{p}_{s_1\ldots s_{2k}}(y_1,y_{k+1}),
    \end{align*}
as $\ve\to0,$ where
\begin{align*}
\tilde{p}_{s_1\ldots s_{2k}}(y_1,y_{k+1})&{}=\tilde{p}_{s_2}\left(Q_{\frac{1}{n},\ldots,1,s_1,s_{k+1},s_2}\left(\tilde{u}\left(\frac{1}{n}\right),\ldots,\tilde{u}(1),y_1,y_{k+1}\right)\right)\times\ldots\\
&{}\quad \quad \times \tilde{p}_{s_{2k}}\left(Q_{\frac{1}{n},\ldots,1,s_1,s_{k+1},s_2\ldots s_{2k}}\left(\tilde{u}\left(\frac{1}{n}\right),\ldots,\tilde{u}(1),y_1,y_{k+1},\ldots\right)\right).
\end{align*}
Now, repeating the same arguments as in the proof of \Cref{thm:Weighted_SILT}, we conclude that
 \begin{align*}
&{} \int_{\mathbb{R}^{2kd-2d}}\prod^{2k-1}_{i=1,i\neq k}f_{\ve}\left(y_{i+1}-Q_{\frac{1}{n},\ldots,s_{i+1}}\left(\tilde{u}\left(\frac{1}{n}\right),\ldots,y_{i}\right)\right)
\prod^{2k}_{i=2,\ i\neq k+1}\tilde{p}_{s_i}(y_i)\differential{y_1}\cdots \differential{y_{2k}}
\leqslant \frac{c}{\prod^{2k}_{i=2,\ i\neq k+1}(\Eof{\tilde{u}_1(s_{i})^2})^{\frac{d}{2}}}.
 \end{align*}  
It follows from the Cauchy inequality and \Cref{lem:moment estimate} that
\begin{align*}
&{} \sup_{y_1,y_{k+1}\in\mbR^d}\E_W \E_{\tilde{u}\left(\frac{1}{n}\right),\ldots,\tilde{u}(1)}\sup_{0\leqslant t\leqslant T}\frac{1}{\lvert\det \D x_n(y_1+a_{s_1}(\tilde{u}\left(\frac{1}{n}\right),\ldots,\tilde{u}(1)),t)\rvert^{k-1}}\\
&{}\quad \quad \quad \times\frac{1}{\lvert\det \D x_n(y_{k+1}+a_{s_{k+1}}(\tilde{u}\left(\frac{1}{n}\right),\ldots,\tilde{u}(1))+a(s_{k+1},s_1)y_1,t)\rvert^{k-1}}
\leqslant C,
\end{align*}
where $C$ is some positive constant. Moreover,  $(2,\zeta)$-local nondeterminism of the process $u$ with $\zeta\in(0,\frac{2}{d})$ allows to conclude that
$$
\int_{\Delta^2_k}\frac{1}{G(u_1(\frac{1}{n}),\ldots,u_1(1),u_1(s_1),\ldots,u_1(s_{2k}))^{\frac{d}{2}}}\differential{s_1}\cdots \differential{s_{2k}}<\infty,
$$
which finishes the proof of the theorem.
\end{proof}

We are now ready to state and prove one of the main results of the paper. The following theorem shows that the self-intersection local times for the process $\{x(u(s),t),\ s\in[0,1]\}$ can be defined as the limit in mean square of random variables $T^{x_n(u,t)}_k, $ as $n \to \infty$. 
  \begin{theorem}
\label{thm:SILT_V_flow}
There exists a random variable $T^{x(u,t)}_k$ such that the sequence $\{T^{x_n(u,t)}_{k}: n \ge 1\}$ converges in mean-square to $T^{x(u,t)}_k$, i.e.,
$
T^{x(u,t)}_k = \llim{n\to\infty}T^{x_n(u,t)}_{k}, 
$
and the following formula holds
$$T^{x(u,t)}_k=T^u_k\left(\frac{1}{\lvert\det\ \D x\rvert^{k-1}}\right),\quad \text{for all}\  k\geqslant 2.$$
\end{theorem}
\begin{proof}[Proof of Theorem~\ref{thm:SILT_V_flow}] We will verify that 
$\Eof{\left(T^{x_n(u,t)}_{k}-T^{x_m(u,t)}_{k}\right)^2}\to 0,
$ as $n, m\to \infty$. 
Note that
$$
\Eof{\left(T^{x_n(u,t)}_{k}-T^{x_m(u,t)}_{k}\right)^2}=\lim_{\ve\to0}\Eof{\left(T^{x_n(u,t)}_{\ve,k}-T^{x_m(u,t)}_{\ve,k}\right)^2}.
$$
Moreover, see that 
 \begin{align*}
&{}\Eof{\left(T^{x_n(u,t)}_{\ve,k}-T^{x_m(u,t)}_{\ve,k}\right)^2}\\
&{}=\E\left(\int_{\Delta_k}\left(\frac{1}{\mid \det \D x_n(u(s_1),t)\mid^{k-1}}-\frac{1}{\mid \det  \D x_m(u(s_1),t)\mid^{k-1}}\right)
\prod^{k-1}_{i=1}f_{\ve}(u(s_{i+1})-u(s_i))\differential{s}_1\cdots \differential{s}_k\right)^2\\
&{}=\E_u\int_{\Delta^2_k}\E_W\left(\frac{1}{\mid \det \D x_n(u(s_1),t)\mid^{k-1}}-\frac{1}{\mid \det  \D x_m(u(s_1),t)\mid^{k-1}}\right)\\
&{}\quad \quad \times\left(\frac{1}{\mid \det \D x_n(u(s_{k+1}),t)\mid^{k-1}}-\frac{1}{\lvert \det  \D x_m(u(s_{k+1}),t)\rvert^{k-1}}\right)
\prod^{2k-1}_{i=1,i\neq k}f_{\ve}(u(s_{i+1})-u(s_i))\differential{s}_1\cdots \differential{s}_{2k}.
\end{align*}
Applying the Cauchy inequality,  \Cref{lem:moments}, and \Cref{thm:detrminant difference estimate},  we conclude that for each $\ve>0$,
\begin{align*}
&{}\Eof{\left(T^{x_n(u,t)}_{\ve,k}-T^{x_m(u,t)}_{\ve,k}\right)^2}
\leqslant c(\E_u\gamma_4(\mu^{(n)},\mu^{(m)})^4)^{\frac{1}{2}}
\left(\E_u\left(\int_{\Delta^2_k}\prod^{2k-1}_{i=1,i\neq k}f_{\ve}(u(s_{i+1})-u(s_i))\differential{s}_1\cdots \differential{s}_{2k}\right)^2\right)^{\frac{1}{2}}.
\end{align*}
It follows from  \Cref{thm:Weighted_SILT_V} that
$$
\lim_{\ve\to0}\E_u\left(\int_{\Delta^2_k}\prod^{2k-1}_{i=1,i\neq k}f_{\ve}(u(s_{i+1})-u(s_i))\differential{s}_1\cdots \differential{s}_{2k}\right)^2=\E_u \left(T_k^u\right)^4.
$$
Hence, we have 
\begin{align*}
&{}\Eof{\left(T^{x_n(u,t)}_{k}-T^{x_m(u,t)}_{k}\right)^2}
\leqslant c(\E_u\gamma_4(\mu^{(n)},\mu^{(m)})^4)^{\frac{1}{2}}\left(\E_u \left(T^u_k\right)^4\right)^{\frac{1}{2}}\to0,\ n,m\to\infty, 
\end{align*}
which completes the proof of the theorem.
\end{proof}


\section{Asymptotics of self-intersection local times} 
\label{sec:Asymptotics SILT}
\cite{baxendale1986isotropic} and \cite{dimitroff2006some} studied the evolution of geometric characteristics of smooth curves (length) under isotropic Brownian flows. 
Let us summarise some interesting results: If $L_t \coloneq \int^1_0\norm{\gamma^{\prime}_t(u)}\differential{u}$ is the length of the curve $\gamma_t \coloneq \phi_t\circ\gamma:[0,1]\to\mbR^d$, then it can be proved that the stochastic process 
$\{\exp\left(-\left(\theta+\frac{\kappa}{2}\right)t\right)L_t : t\ge 0\}$ is a martingale, which converges almost surely as $t\to \infty$, where 
$\phi_t$ is an isotropic Brownian flow, $\theta$ is the top Lyapunov exponent associated with the flow,  and $\kappa$ is the characteristic constant of the flow. See \cite{dimitroff2006some,le1985isotropic,baxendale1986isotropic,kunita1990stochastic} for details and precise definitions.

What happens if the curve is nonsmooth and random? Following Le Gall's approach, the self-intersection local times can be considered its geometric characteristics. The main aim of this section is to describe its asymptotics under the action of stochastic flows. To be more precise, we shall prove that $\{\exp\left((k-1)\hat{a}t-\frac{k(k-1)\hat{b}t}{2}\right)T^{x(u,t)}_k : t\ge 0\}$ is a positive, continuous square-integrable martingale for an appropriate choice of $\hat{a}$, and $\hat{b}$ (\Cref{thm:silt martingale}). We provide an explicit expression for its quadratic variation.


Let $u$ be the Volterra Gaussian process as before.  Consider the stochastic flow generated by the solution to the following stochastic differential equation with interaction
\begin{align}
\label{eq:special_eq_interaction}
\begin{cases}
\differential{x(v,t)}=a(x(v,t),\mu_t)\differential{t}+\int_{\mbR^d}b(x(v,t)-z)W(\differential{t},\differential{z})\\	
x(v,0)=v,\ v\in\mbR^d\\
\mu_t=\mu_0\circ x(\cdot,t)^{-1},\ t\geqslant 0.
\end{cases}
\end{align}
Assume that coefficients $a$ and $b$ with $b(x, \nu, z)\equiv b(x-z)$ satisfy the conditions of  \Cref{thm:existence_solution} and \Cref{thm:determinant}. 
Then,  $\{x(v,t):  v\in\mbR^d,\ t\in[0,\infty)\}$ is the stochastic flow of diffeomorphisms.
Note that if $a=0$ and $b(y)=\tilde{b}(\norm{y}),\ y\in\mbR^d,$ for some $\tilde{b}$, then the solution is the isotropic Brownian flow with covariance $(t_2\wedge t_1)\ b*b(y_2-y_1),$ where $b*b$ denotes the convolution of $b$ with itself. 

Let us define
$$
\hat{b}=\trace\ \int_{\mbR^d}\D b(z)\odot \D b(z)\differential{z}, 
$$
where $A\odot B$ denotes  the Hadamard product of matrices $A$ and $B$. 
The following theorem describes the asymptotics of the self-intersection local times of the stochastic process $\{x(u(s)): s \in [0, 1]\}$ as $t\to\infty.$ 
\begin{theorem}
\label{thm:asymptotics_expectations} Assume that $\trace\ \D a(v,\mu)=\hat{a}$ for all $v\in\mbR^d$ and $\mu\in\cM_m.$ Then, for all $k \ge 2$, 
$$
\lim_{t\to\infty}\exp\left((k-1)\hat{a}t-\frac{k(k-1)\hat{b}t}{2}\right)\Eof{T^{x(u,t)}_k}=\Eof{T^u_k}.
$$
\end{theorem}
\begin{proof}[Proof of \Cref{thm:asymptotics_expectations}] 
It follows from \Cref{thm:determinant} that 
\begin{align*}
&{}\frac{1}{\lvert \det\ \D x(v,t)\rvert^{k-1}}
=\exp\left(-(k-1)\hat{a}t\frac{(k-1)\hat{b}t}{2}-(k-1)\trace\int^t_0\int_{\mbR^d}\D b(x(v,s)-z)W(\differential{s},\differential{z})\right).
\end{align*}
Note that the stochastic process
$$
\beta_k(v,t)\coloneq -(k-1)\trace\int^t_0\int_{\mbR^d}\D b(x(v,s)-z)W(\differential{s},\differential{z})
$$
is a continuous square integrable martingale with respect to 
$
\tilde{\cF}_t=\sigma\left(u(r), W(\Delta):  r\in[0,1],\ \Delta\in \cB\left([0,t]\times\mbR^d\right)\right),
$
which implies that
$$
\cE_k(v,t)\coloneq \myExp{\beta_k(v,t)-\frac{(k-1)^2\hat{b}t}{2}}, t\ge 0, v\in\mbR^d,
$$
is also a continuous square integrable martingale with respect to $\tilde{\cF}_t.$ Moreover,  $\Eof{\cE_k(v,t)}=1.$ 
Therefore, 
\begin{align*}
\exp\left((k-1)\hat{a}t-\frac{k(k-1)\hat{b}t}{2}\right)\Eof{T^{x(u,t)}_k}
&{}=\Eof{\int_{\Delta_k}\cE_k(u(s_1),t)\prod^{k-1}_{i=1}\delta_0(u(s_{i+1})-u(s_i))\differential{s_1}\cdots \differential{s_k}}\\
&{}=\E_u\int_{\Delta_k}\prod^{k-1}_{i=1}\delta_0(u(s_{i+1})-u(s_i))\differential{s_1}\cdots \differential{s_k},
\end{align*}
which finishes the proof of the theorem.
\end{proof}
\Cref{thm:asymptotics_expectations} describes the asymptotics of the mean  of the random variable $T^{x(u,t)}_k$. Our next aim is to study its almost sure asymptotics. To do this, let us introduce a random measure on $\mbR^d$ defined as follows
$$
\nu_k(A)\coloneq \int^1_{0}\cdots\int^1_{0}1_{A}(u(s_1))\prod^{k-1}_{i=1}\delta_0(u(s_{i+1})-u(s_i))\differential{s_1}\cdots \differential{s_k},\ A\in\cB(\mbR^d).
$$
Note that the existence of such a measure follows from  \Cref{thm:Weighted_SILT_V}. Moreover, it follows from Theorem \ref{thm:Weighted_SILT_V} that $\Eof{\nu_k(\mbR^d)^p}<\infty$  for each $p\geqslant 2.$ 
Then, define the stochastic process
\begin{align*}
  \begin{aligned}
  \tilde{T}^{x(u,t)}_k \coloneq \int^1_0\cdots\int^1_0\prod^{k-1}_{i=1}\delta_0(x(u(s_{i+1}),t)-x(u(s_i),t))\differential{s_1}\cdots\differential{s_k} &{} = \int_{\mbR^d}\frac{1}{\mid \D x(v,t)\mid^{k-1}}\nu_k(\differential{v}),\\
\text{and}\quad  \cT_k(t)\coloneq \exp\left((k-1)\hat{a}t-\frac{k(k-1)\hat{b}t}{2}\right)\tilde{T}^{x(u,t)}_k
&{}
=\int_{\mbR^d}\cE_k(v,t)\nu_k(\differential{v}).
  \end{aligned}
\end{align*}


The following lemma shows that the random process $\{\cT_k(t) : t\ge 0\}$ is, in fact, a square integrable martingale and describes its quadratic variation.
\begin{lemma}
\label{thm:silt martingale} The random process $\{\cT_k(t): t \geqslant 0\}$ is a positive continuous square integrable martingale with respect to $\tilde{\cF}_t$ with quadratic variation
\begin{align*}
&{}\langle\cT_k\rangle(t)=(k-1)^2\int_{\mbR^d}\int_{\mbR^d}\int^t_0\cE_k(v_1,s)\cE_k(v_2,s)
\left(\sum^d_{i=1}\D b_{ii}*\D b_{ii}(x(v_2,s)-x(v_1,s))\right)\differential{s} \nu_k(\differential{v_1})\nu_k(\differential{v_2}).
\end{align*}
\end{lemma}
\begin{proof}[Proof of \Cref{thm:silt martingale}]
The positivity and the martingale property follows from the definition of $\cT_k$ and the measurability of the random measure $\nu_k$ with respect to $\tilde{\cF}_0\subset\tilde{\cF}_t.$ Moreover, 
$\Eof{\cT_k(t)^2}<\infty$ for each $t\geqslant 0.$
Note that
\begin{align*}
\langle\cT_k\rangle(t)&{}=\int_{\mbR^d}\int_{\mbR^d}\langle\cE_k(v_1,\cdot)\cE_k(v_2,\cdot)\rangle(t)\nu_k(\differential{v_1})\nu_k(\differential{v_2})
\\ 
&{}
=\int_{\mbR^d}\int_{\mbR^d}\int^t_0\cE_k(v_1,s)\cE_k(v_2,s)\differential{\langle\beta_k(v_1,\cdot)\beta_k(v_2,\cdot) \rangle(s)}\nu_k(\differential{v_1})\nu_k(\differential{v_2}).
\end{align*}
This completes the proof since  $\langle \beta_k(v_1,\cdot)\beta_k(v_2,\cdot)\rangle(t) = (k-1)^2 \sum^d_{i=1}\int^t_0\D b_{ii} * \D b_{ii}(x(v_2,s)-x(v_1,s))\differential{s}.$
\end{proof}

Since the process $\{\cT_k(t) : t\ge 0\}$ is a continuous square-integrable martingale, by \cite[Chapter V, Proposition 1.8, p. 183]{RevuzYor1999CMBMbook}, the two sets $\{\lim_{t\to \infty} \cT_k(t) \}$ and $\{ \langle\cT_k\rangle(\infty)\equiv \lim_{t\to \infty} \langle\cT_k\rangle(t)  < \infty\}$ are almost surely equal. Therefore, when $\trace\ \D a(v,\mu)=\hat{a}$ for all $v\in\mbR^d$ and $\mu\in\cM_m$, we conclude 
\begin{align}
\label{eq:silt asymptotics a.s.}
\lim_{t\to\infty}\exp\left((k-1)\hat{a}t-\frac{k(k-1)\hat{b}t}{2}\right)\tilde{T}^{x(u,t)}_k\in[0,\infty)\ \text{a.s.}
\end{align}
on the event $\{ \lim_{t\to \infty} \langle\cT_k\rangle(t)  < \infty\}$. The following lemma provides an estimate on $\langle\cT_k\rangle$.

\begin{lemma}
\label{thm:characteristics estimate} There exists a positive $c$ such that
$
 \limsup_{t\to\infty}\myExp{\frac{-3(k-1)^2\hat{b}t}{2}}\Eof{\langle\cT_k\rangle(t)}\leq c
$
for all $k\ge 2$. 
\end{lemma}
\begin{proof}[Proof of \Cref{thm:characteristics estimate}] Note that
\begin{align*}
\Eof{\langle\cT_k\rangle(t)}&{} \leq c(k-1)^2\Eof{\int_{\mbR^d}\int_{\mbR^d}\int^t_0\cE(v_1,s)\cE(v_2,s)\differential{s}\ \nu_k(\differential{v}_1)\nu_k(\differential{v}_2)}\\
&{}\leq c(k-1)^2\int_{\mbR^d}\int_{\mbR^d}\int^t_0(\Eof{\cE(v_1,s)^2})^{\frac{1}{2}}(\Eof{\cE(v_2,s)^2})^{\frac{1}{2}}\differential{s}\ \nu_k(\differential{v}_1)\nu_k(\differential{v}_2).
\end{align*}
Since 
$
\Eof{\cE(v_i,s)^2}=\myExp{\frac{3(k-1)^2\hat{b}s}{2}},\ i=1,2,
$
we have, for some positive constant $c$,
$$
\Eof{\langle\cT_k\rangle(t)}\leq c \myExp{\frac{3(k-1)^2\hat{b}t}{2}}.
$$
\end{proof}

\appendix

 \section{Stochastic flows}
  \label{sec: Stochastic flows}
  \begin{definition}
 \label{def:stochastic flows} A family of random maps $\{\phi_{s,t}:\mbR^d\times\Omega\to\mbR^d,\ 0\leqslant s\leqslant t<\infty\}$ is called a stochastic flow of homeomorphisms if for $\prob$-almost all $\omega\in\Omega:$  
 \begin{enumerate}[label=(\roman*)]
 \item $\phi_{s,t}(\omega)=\phi_{r,t}(\omega)\circ\phi_{s,r,}(\omega)\ \text{holds for all}\ s\leqslant r\leqslant t.$ 
 \item $\phi_{s,s}(\omega)$ is the identity map for all $s\geqslant 0.$
 \item $\phi_{s,t}(\omega):\mbR^d\to\mbR^d$ is a homeomorphism for all $s\leqslant t.$ 
 \end{enumerate}
 If additionally, the mapping  $\phi_{s,t}(\omega):\mbR^d\to\mbR^d$ is $k$-times continuously differentiable for all $s\leqslant t,$ then $\phi$ is a stochastic flow of $C^k$-diffeomorphisms.
  \end{definition}

  We refer the readers to \cite{Darling1992isotropic,baxendale1986isotropic,le1985isotropic,kunita1990stochastic,dimitroff2006some} for a rigorous exposition on stochastic flows.

\section*{Acknowledgements}
Olga Izyumtseva was  supported by the British Academy through grant number RaR{\textbackslash}100741, and in part by British Academy, Cara, Leverhulme Trust through grant LTRSF24{\textbackslash}100014.

\printcredits

\bibliographystyle{cas-model2-names}

\bibliography{refs}



\end{document}